\documentclass[oneside,american]{amsart}
\usepackage[T1]{fontenc}
\usepackage[latin9]{inputenc}
\usepackage{textcomp}
\usepackage{amstext}
\usepackage{amsthm}
\usepackage{amssymb}

\makeatletter
\numberwithin{equation}{section}
\numberwithin{figure}{section}
\theoremstyle{plain}
\newtheorem{thm}{\protect\theoremname}
\theoremstyle{definition}
\newtheorem{defn}[thm]{\protect\definitionname}
\theoremstyle{definition}
\newtheorem{example}[thm]{\protect\examplename}
\theoremstyle{plain}
\newtheorem{prop}[thm]{\protect\propositionname}
\theoremstyle{remark}
\newtheorem*{rem*}{\protect\remarkname}
\theoremstyle{plain}
\newtheorem{lem}[thm]{\protect\lemmaname}
\theoremstyle{plain}
\newtheorem{cor}[thm]{\protect\corollaryname}
\theoremstyle{remark}
\newtheorem{rem}[thm]{\protect\remarkname}

\usepackage[all]{xy}

\makeatother

\usepackage{babel}
\providecommand{\corollaryname}{Corollary}
\providecommand{\definitionname}{Definition}
\providecommand{\examplename}{Example}
\providecommand{\lemmaname}{Lemma}
\providecommand{\propositionname}{Proposition}
\providecommand{\remarkname}{Remark}
\providecommand{\theoremname}{Theorem}

\begin{document}

\title{On the structure of certain $\Gamma$-difference modules}

\author{Ehud de Shalit and José Gutiérrez}

\address{Ehud de Shalit, Einstein Institute of Mathematics, Hebrew University
of Jerusalem, Israel.}

\email{ehud.deshalit@mail.huji.ac.il}

\address{José-Ibrahim Villanueva-Gutiérrez, Einstein Institute of Mathematics,
Hebrew University of Jerusalem, Israel.}

\email{jose.gutierrez@mail.huji.ac.il}
\begin{abstract}
This is a largely expository paper, providing a self-contained account
on the results of \cite{Sch-Si1,Sch-Si2}, in the cases denoted there
2Q and 2M. These papers of Schäfke and Singer supplied new proofs
to the main theorems of \cite{Bez-Bou,Ad-Be}, on the rationality
of power series satisfying a pair of independent $q$-difference,
or Mahler, equations. 

We emphasize the language of $\Gamma$-difference modules, instead
of difference equations or systems. Although in the two cases mentioned
above this is only a semantic change, we also treat a new case, which
may be labeled 1M1Q. Here the group $\Gamma$ is generalized dihedral
rather than abelian, and the language of equations is inadequate.

In the last section we explain how to generalize the main theorems
in case 2Q to finite characteristic.
\end{abstract}


\maketitle

\section*{Introduction}

Adamczewski and Bell proved in 2017 the following theorem, conjectured
some 30 years earlier by Loxton and van der Poorten \cite{vdPo}.
\begin{thm}
\cite{Ad-Be} \label{thm:L-vdP} Let $p$ and $q$ be multiplicatively
independent natural numbers. Consider the endomorphisms
\[
\sigma f(x)=f(x^{p}),\,\,\,\tau f(x)=f(x^{q})
\]
of the field of rational functions $K=\mathbb{C}(x)$ and of its completion
at 0, the field of Laurent series $k=\mathbb{C}((x))$. If $f\in k$
satisfies the two \emph{Mahler equations
\[
\sum_{i=0}^{n}a_{i}\sigma^{n-i}(f)=0,\,\,\,\sum_{i=0}^{m}b_{i}\tau^{m-i}(f)=0,
\]
with $a_{i},b_{i}\in K$, then $f\in K.$}
\end{thm}

For an account on Mahler's equations and their role in transcendence
theory, see the survey paper \cite{Ad}. A similar theorem has been
proved by Bézivin and Boutabaa in 1992.
\begin{thm}
\cite{Bez-Bou} \label{thm:Bezivin}Let $p$ and $q$ be multiplicatively
independent complex numbers. Consider the automorphisms
\[
\sigma f(x)=f(px),\,\,\,\tau f(x)=f(qx)
\]
of the fields $K=\mathbb{C}(x)$ and $k=\mathbb{C}((x)).$ If $f\in k$
satisfies the two $q$\emph{-difference} equations
\[
\sum_{i=0}^{n}a_{i}\sigma^{n-i}(f)=0,\,\,\,\sum_{i=0}^{m}b_{i}\tau^{m-i}(f)=0
\]
\emph{with $a_{i},b_{i}\in K$, then $f\in K.$}
\end{thm}

The proofs of these two theorems used a variety of techniques. In
the case of Theorem \ref{thm:L-vdP} it relied on Cobham's theorem
\cite{Co} in the theory of automata. Theorem \ref{thm:Bezivin} was
proved by $p$-adic techniques, for an auxiliary prime\footnote{Having nothing to do with the complex number denoted by $p$ in the
statement of the theorem.} $p$, and involved, in its original formulation, some unnecessary
restrictions. Recently, Schäfke and Singer \cite{Sch-Si1,Sch-Si2}
provided a uniform treatment of the two theorems, as well as of other
similar results. Besides emphasizing common features, they eliminated
the dependence, in the work of Adamczewski and Bell, on Cobham's theorem.
In fact, the latter could now be deduced from Theorem \ref{thm:L-vdP}.
In Theorem \ref{thm:Bezivin} they still required $p$ or $q$ to
be of absolute value different than 1 (at least after some automorphism
of the complex numbers), but this restriction can be removed.

\bigskip{}

The goal of this largely expository paper is to provide yet another
look at the same theorems, leading to a third, new example. We give
a self-contained treatment, based on the notion of a $\Gamma$-difference
module, which is introduced in Section \ref{sec:-=00005CGamma difference},
and we shift the focus from \emph{equations} to\emph{ modules}. This
is similar to studying linear partial differential equations via $D$-modules.
The letter $\Gamma$ signifies the group of automorphisms (of $K$
or of some extension field $\widetilde{K}$) generated by the operators
$\sigma$ and $\tau,$ which, in the two examples cited above, is
free abelian of rank 2. This approach allows us to isolate, in Section
\ref{sec:Formal}, the formal aspects of the theory. Once we globalize,
in Section \ref{sec:Rational}, our proof of Theorem \ref{thm:L-vdP}
follows the line of \cite{Sch-Si1,Sch-Si2}. In Theorem 2 we remove
the unnecessary restriction that $|p|\ne1$ or $|q|\ne1$ (see Step
VI of \S\ref{subsec:(p,q)-difference modules}).

We treat the above two theorems, corresponding to the cases 2M and
2Q in \cite{Sch-Si1}. In Section \ref{sec:-1M1Q} we give a third
example that might be denoted 1M1Q. In this case the group $\Gamma$
is no longer abelian, but rather \emph{generalized dihedral}. As a
result, the main theorem does not lend itself to a simple-minded formulation
in terms of equations as above, but its formulation (and proof) in
the language of difference modules is completely analogous to the
first two cases.

In the last section we explain how to generalize Theorem \ref{thm:Bezivin},
as well as case 2Q of our Main Theorem (Theorem \ref{thm:Main Theorem}),
if an arbitrary field of constants, possibly of finite characteristic,
is substituted for the field of complex numbers.

Finally we remark that in \cite{dS1,dS2} a similar situation, not
discussed in the present paper, is studied, where the field of rational
functions is replaced by a field of elliptic functions. New issues
arise there. One is the issue of periodicity. Another one is the existence
of non-trivial $\Gamma$-invariant vector bundles on the elliptic
curve. Nevertheless, these issues can be analyzed, and a theorem analogous
to the two theorems cited above, where the operators $\sigma$ and
$\tau$ are induced by isogenies of the elliptic curve, and the coefficients
$a_{i}$ and $b_{i}$ are elliptic functions, is proved there. We
stress that in the elliptic case, a power series $f$ satisfying two
elliptic $p$- and $q$-difference equations need not be elliptic.
Instead, it belongs to a slightly larger ring of functions, generated
over the field of elliptic functions (in the variable $z$) by $z,z^{-1}$
and the Weierstrass zeta function of $z$.

Recent work of Adamczewski, Dreyfus, Hardouin and Wibmer established
a far-reaching strengthening of the above mentioned theorems. In \cite{ADHW}
they show that if $f,g\in\mathbb{C}((x))$ do not belong to $\mathbb{C}(x)$,
$f$ satisfies a $\sigma$-difference equation and $g$ a $\tau$-difference
equation, then $f$ and $g$ are algebraically independent over $\mathbb{C}(x).$
Special cases of this result have been proved by various authors before.
Our survey raises two immediate questions: (a) Can a generalization
of this type be phrased (and proved) in the context of \emph{difference
modules}, that will apply for example in the case 1M1Q, not amenable
to a formulation in terms of two difference equations? (b) Can the
prerequisites for a such a theorem be axiomatized (and checked) to
include, for example, ground fields of elliptic functions?

\section{$\Gamma$-difference modules\label{sec:-=00005CGamma difference}}

\subsection{Definitions and examples}

Let $K$ be a field and $\Gamma$ a group, acting on $K$ by automorphisms.
We make no assumption whatsoever on the nature of $\Gamma$, nor do
we require the action to be faithful. In fact, the case of a trivial
action is not excluded. The fixed field $C=K^{\Gamma}$ is called
the \emph{field of constants}.
\begin{defn}
A $\Gamma$\emph{-difference module} over $K$ is a finite dimensional
$K$-vector space $M$ equipped with a semi-linear action of $\Gamma$.
In other words, for every $\gamma\in\Gamma$ there is a $\Phi_{\gamma}\in GL_{C}(M)$
satisfying
\[
\Phi_{\gamma}(av)=\gamma(a)\Phi_{\gamma}(v)\,\,\,(a\in K,\,v\in M)
\]
and
\[
\Phi_{\gamma\delta}=\Phi_{\gamma}\circ\Phi_{\delta}\,\,\,(\gamma,\delta\in\Gamma).
\]
\end{defn}

If the action of $\Gamma$ on $K$ is trivial, this notion is nothing
but a linear representation of $\Gamma$ over $K$. If $K_{0}\subset K$
is a $\Gamma$-invariant subfield then we say that $M$ \emph{descends
to} $K_{0}$, or has an underlying $K_{0}$\emph{-structure}, if there
exists a $\Gamma$-difference module $M_{0}$ over $K_{0}$ such that
$M\simeq K\otimes_{K_{0}}M_{0},$ the $\Gamma$-action extended semi-linearly.
This may apply in particular to $K_{0}=C.$ In general, if $M$ descends
to $K_{0},$ $M_{0}$ need not be unique, not even up to an isomorphism
over $K_{0}$. If $M$ descends to the trivial module over $C$, i.e.
if $M\simeq K^{r}$, $\Gamma$ acting in the coordinates, we call
$M$ \emph{trivial.}
\begin{example}
If $\Gamma$ is a finite group acting faithfully then $K/C$ is a
Galois extension with $Gal(K/C)=\Gamma$, and Hilbert's theorem 90
says that every $\Gamma$-difference module over $K$ is trivial.
\end{example}

Denoting by $K\left\langle \Gamma\right\rangle $ the twisted group
ring of $\Gamma$ over $K$, a $\Gamma$-difference module is nothing
but a $K\left\langle \Gamma\right\rangle $-module, finite dimensional
over $K$. The category of $\Gamma$-difference modules over $K$
will be denoted $\Gamma\mathrm{Diff}_{K}$. The tensor product $M\otimes N$
of two $\Gamma$-difference modules is defined as the tensor product
over $K$, with the usual $\Gamma$-action, $\Phi_{\gamma}(u\otimes v)=\Phi_{\gamma}(u)\otimes\Phi_{\gamma}(v).$
The dual $M^{\vee}$ is defined as the space of $K$-linear functionals
$\lambda:M\to K$ with the action
\[
\Phi_{\gamma}(\lambda)(v)=\gamma(\lambda(\Phi_{\gamma}^{-1}v)),
\]
and the internal hom is $\underline{Hom}(M,N)=M^{\vee}\otimes N.$

It is easily checked that with these definitions $\Gamma\mathrm{Diff}_{K}$
becomes a rigid abelian tensor category (\cite{De-Mi}, Definition
1.15). The object $\underline{1}$ is the trivial $\Gamma$-module
$K$ and $End(\underline{1})=C$. The Tannakian formalism applies
to this category, as described for example in the last section of
the first chapter of \cite{vdP-Si}, but we shall not dwell on this
aspect here.

The following easy Proposition should be viewed as a generalization
of Hilbert's theorem 90.
\begin{prop}
Suppose that
\[
1\to\Delta\to\Gamma\to\overline{\Gamma}\to1
\]
is a short exact sequence of groups, $\Delta$ is finite, and the
action of $\Delta$ on $K$ is faithful. Let $K_{0}=K^{\Delta}.$
Then the categories $\Gamma\mathrm{Diff}_{K}$ and $\overline{\Gamma}\mathrm{Diff}_{K_{0}}$
are equivalent.
\end{prop}

\begin{proof}
Consider the two functors $\alpha:\Gamma\mathrm{Diff}_{K}\to\overline{\Gamma}\mathrm{Diff}_{K_{0}}$
and $\beta:\overline{\Gamma}\mathrm{Diff}_{K_{0}}\to\Gamma\mathrm{Diff}_{K}$
defined by
\[
\alpha(M)=M^{\Delta},\,\,\,\beta(M_{0})=K\otimes_{K_{0}}M_{0}.
\]
Since $\Delta$ is normal in $\Gamma$, the action of $\Gamma$ on
$K$ induces an action of $\overline{\Gamma}$ on $K_{0}$ and $\alpha(M)$
becomes a $\overline{\Gamma}$-difference module. Likewise, $\beta(M_{0})$
becomes a $\Gamma$-difference module if the action of $\Gamma$ on
$M_{0}$ (which factors through $\overline{\Gamma}$) is extended
semi-linearly to $K\otimes_{K_{0}}M_{0}$. Hilbert's Theorem 90 says
that when we restrict the action from $\Gamma$ to $\Delta$ these
two functors give an equivalence beween $\Delta\mathrm{Diff}_{K}$
and $\overline{\Delta}\mathrm{Diff}_{K_{0}},$ the latter being the
category of finite dimensional vector spaces over $K_{0}$. In particular
$\dim_{K_{0}}\alpha(M)=\dim_{K}(M)$ and $\dim_{K}\beta(M_{0})=\dim_{K_{0}}(M_{0}).$
Galois theory gives $\alpha\circ\beta(M_{0})=M_{0}.$ On the other
hand there is an injective map $K\otimes_{K_{0}}\alpha(M)\to M$ respecting
the action of $\Gamma,$ so by dimension counting it must be an isomorphism
and we also have $\beta\circ\alpha(M)=M.$
\end{proof}
Thus, when studying $\Gamma$-difference modules over a field $K$,
we may always factor out finite normal subgroups of $\Gamma$, if
they act faithfully on $K$. For example, if $\Gamma$ is a semisimple
algebraic group acting faithfully on $K$, we may assume, without
loss of generality, that it is of adjoint type.

\subsection{Matrices and classification\label{subsec:Matrices}}

If we choose a basis $e_{1},\dots,e_{r}$ of $M$ over $K$ we may
associate to any $\gamma\in\Gamma$ its matrix $(a_{ij})$, defined
by
\[
\Phi_{\gamma}(e_{j})=\sum_{i=1}^{r}a_{ij}e_{i}.
\]
It is customary to denote by $A_{\gamma}$ the \emph{inverse} of this
matrix, namely $A_{\gamma}^{-1}=(a_{ij}).$ The condition $\Phi_{\gamma}\circ\Phi_{\delta}=\Phi_{\gamma\delta}$
gets translated to the \emph{consistency condition
\[
A_{\gamma\delta}=\gamma(A_{\delta})\cdot A_{\gamma},
\]
}which must hold for every $\gamma,\delta\in\Gamma.$ Conversely,
a collection of matrices $\{A_{\gamma}\}$ satisfying the above conditions,
termed a \emph{consistent collection of matrices}, defines a $\Gamma$-difference
module structure on $K^{r}$ by letting 
\[
\Phi_{\gamma}(v)=A_{\gamma}^{-1}\gamma(v).
\]

If $e'_{1},\dots,e'_{r}$ is another basis and $C=(c_{ij})$ is the
transition matrix, i.e.
\[
e'_{j}=\sum_{i=1}^{r}c_{ij}e_{i},
\]
then the matrix $A'_{\gamma}$ corresponding to $\gamma$ in the new
basis is
\begin{equation}
A'_{\gamma}=\gamma(C)^{-1}A_{\gamma}C.\label{eq:gauge}
\end{equation}
The equivalence relation defined by $(\ref{eq:gauge})$ is called
\emph{gauge equivalence. }It follows that $\Gamma$-difference modules
of rank $r$ over $K$ are classified by gauge equivalence classes
of consistent collections of matrices $A_{\gamma}\in GL_{r}(K)$,
or what is the same, by the non-abelian cohomology
\[
H^{1}(\Gamma,GL_{r}(K)).
\]

\begin{example}
\label{exa:examples}(i) $\Gamma=\left\langle \sigma\right\rangle $
is infinite cyclic. In this case a consistent collection is determined
uniquely by $A_{\sigma}$, which may be chosen arbitrarily, and the
gauge equivalence classes correspond to the twisted conjugacy classes
\[
B(GL_{r}(K))=GL_{r}(K)/\sim
\]
where $A'\sim A$ if there exists a $C\in GL_{r}(K)$ with $A'=\sigma(C)^{-1}AC.$

(ii) Let $k$ be a perfect field of characteristic $p$ and $K=W(k)[1/p]$
where $W(k)$ is the ring of Witt vectors of $k$. Let $\sigma$ denote
the Frobenius automorphism of $K$ and $\Gamma=\left\langle \sigma\right\rangle .$
A $\Gamma$-difference module over $K$ is called also an $F$\emph{-isocrystal}.
This notion is central to $p$-adic Hodge theory.

(iii) Replacing the group $GL_{r}(K)$ by $G(K)$ for an arbitrary
linear algebraic group $G$ over $K,$ one arrives at the notion of
a \emph{$\Gamma$-difference module with $G$-structure.} These objects
are classified by $H^{1}(\Gamma,G(K))$, and when $\Gamma=\left\langle \sigma\right\rangle $
by $B(G(K)),$ defined as above. In example (ii) they have been analyzed
in \cite{Kot}.

(iv) $\Gamma=\left\langle \sigma,\tau\right\rangle \simeq\mathbb{Z}^{2}$
(i.e. $\sigma$ and $\tau$ commute and are multiplicatively independent:
$\sigma^{a}\tau^{b}=1$ if and only if $a=b=0$). In this case a $\Gamma$-difference
module is defined by the pair $(A_{\sigma},A_{\tau}),$ subject to
the consistency condition
\begin{equation}
\sigma(A_{\tau})A_{\sigma}=\tau(A_{\sigma})A_{\tau},\label{eq:consistency}
\end{equation}
up to gauge equivalence. This is the example underlying the two theorems
cited in the introduction.
\end{example}

\subsection{Difference modules and difference equations}

From now on let, as in the intoduction,
\[
K=\mathbb{C}(x),\,\,\,k=\mathbb{C}((x)).
\]
To give a uniform treatment of Theorem \ref{thm:L-vdP} (case 2M)
and Theorem \ref{thm:Bezivin} (case 2Q) we introduce also the fields
\[
\widetilde{K}=\bigcup_{s=1}^{\infty}\mathbb{C}(x^{1/s})
\]
and
\[
\widetilde{k}=\bigcup_{s=1}^{\infty}\mathbb{C}((x^{1/s})).
\]
The field $\widetilde{k}$ is the field of Puiseux series, and is
the algebraic closure of $k.$ 

In both theorems, $\sigma$ and $\tau$ are endomorphisms of the algebraic
group $\mathbb{G}=\mathbb{G}_{m,\mathbb{C}}$ or $\mathbb{G}_{a,\mathbb{C}},$
and can be extended to \emph{automorphisms} of its universal covering
$\mathbb{\widetilde{G}}$. In the $q$-difference case (2Q) the additive
group is simply connected, so $\mathbb{G}=\widetilde{\mathbb{G}}$.
In the Mahler case (2M) the extension of $\sigma$ or $\tau$ to an
automorphism of $\widetilde{\mathbb{G}}$ depends on the choice of
a compatible sequence of $s$th roots of the function $x$, namely
$\sigma(x^{1/s})=x^{p/s}$ and $\tau(x^{1/s})=x^{q/s}.$ We fix such
a choice once and for all. Replacing $x^{1/s}$ by $\zeta_{s}x^{1/s}$
where $\zeta_{s}$ is an $s$th root of 1 (and, to maintain the compatibility,
$\zeta_{st}^{t}=\zeta_{s}$) results in twisting the action of $\sigma$
on $x^{1/s}$ by $\zeta_{s}^{p-1}$ and the action of $\tau$ on the
same element by $\zeta_{s}^{q-1}.$ The field $\widetilde{K}$ is
the function field of $\mathbb{\widetilde{G}},$ and $\sigma$ and
$\tau$ induce automorphisms of $\widetilde{K}$ and of $\widetilde{k}.$ 

In both cases we therefore let
\[
\Gamma=\left\langle \sigma,\tau\right\rangle ,
\]
acting via automorphisms on the fields $K,k$ in case 2Q, and on the
fields $\widetilde{K},\widetilde{k}$ in case 2M. The significance
of the assumption on the multiplicative independence of $p$ and $q$
is that $\Gamma\simeq\mathbb{Z}^{2}$. 
\begin{thm}[Main Theorem]
\label{thm:Main Theorem}  In either case 2Q or case 2M, any $\Gamma$-difference
module $M$ over $K$ (in case 2Q) or $\widetilde{K}$ (in case 2M)
has an underlying $\mathbb{C}$-structure $M_{0},$ i.e. there exists
a $\Gamma$-invariant $\mathbb{C}$-submodule $M_{0}\subset M$, such
that $M=K\otimes_{\mathbb{C}}M_{0}$ (case 2Q), or $M=\widetilde{K}\otimes_{\mathbb{C}}M_{0}$
(case 2M).
\end{thm}

\begin{rem*}
(i) An equivalent formulation is that any pair $(A_{\sigma},A_{\tau})$
of matrices from $GL_{r}(K)$ (resp. $GL_{r}(\widetilde{K})$) satisfying
the consistency equation $(\ref{eq:consistency})$ is gauge-equivalent
to a pair $(A_{\sigma}^{0},A_{\tau}^{0})$ of constant matrices from
$GL_{r}(\mathbb{C})$.

(ii) Equivalently, the natural map $H^{1}(\Gamma,GL_{r}(\mathbb{C}))\to H^{1}(\Gamma,GL_{r}(K))$
(resp. $H^{1}(\Gamma,GL_{r}(\widetilde{K}))$) is surjective.

(iii) In case 2M the underlying complex structure is \emph{unique},
equiv. the pair $(A_{\sigma}^{0},A_{\tau}^{0})$ is unique up to conjugation
in $GL_{r}(\mathbb{C})$, equiv. the map $H^{1}(\Gamma,GL_{r}(\mathbb{C}))\to H^{1}(\Gamma,GL_{r}(K))$
is bijective. In the case 2Q this is false, already in rank 1. See
remark \ref{rem:Resonants}.

(iv) Note that in the formulation of the last theorem the field $k$
or $\widetilde{k}$ plays no role. It will, however, reappear in its
proof. Note also that the formulation of the theorem is purely algebraic.
By this we mean that if $\iota$ is an arbitrary automorphism of $\mathbb{C}$
and $M$ is a $\Gamma$-difference module, then so is the module $M^{\iota}=\mathbb{C}\otimes_{\iota,\mathbb{C}}M$
obtained from it by transport of structure, and $M$ descends to $\mathbb{C}$
if and only $M^{\iota}$ descends to $\mathbb{C}.$ The topological
or dynamical nature of $M^{\iota}$ may nevertheless be completely
different, as $\iota$ is, in general, non-continuous.
\end{rem*}
\begin{prop}
\label{prop:Modules_and_equations}Theorem \ref{thm:Main Theorem}
implies Theorem \ref{thm:L-vdP} and Theorem \ref{thm:Bezivin}.
\end{prop}

\begin{proof}
Observe first that in the case 2M, to prove Theorem \ref{thm:L-vdP}
it is enough to prove the analogous theorem with $K$ and $k$ replaced
by $\widetilde{K}$ and $\widetilde{k}$, where (the extended) $\sigma$
and $\tau$ become automorphisms. This is because $k\cap\widetilde{K}=K.$
To unify the notation, in this proof only, we let the symbols $k$
and $K$ stand, in case 2M, for the fields $\widetilde{k}$ and $\widetilde{K}$.

Let $M\subset k$ be the $K\left\langle \Gamma\right\rangle $-span
of $f,$ and let $\Phi_{\sigma}=\sigma$ and $\Phi_{\tau}=\tau.$
The condition imposed on $f,$ that it simultaneously satisfies the
two functional equations with coefficients from $K$, is equivalent
to the condition
\[
\dim_{K}M<\infty.
\]
Indeed, thanks to the commutativity of $\Gamma,$ $M$ is spanned
by $\sigma^{i}\tau^{j}f$ for $0\le i<n$ and $0\le j<m.$

Let $e_{1},\dots,e_{r}$ be a basis over $\mathbb{C}$ of the submodule
$M_{0}$, whose existence is guaranteed by Theorem \ref{thm:Main Theorem}.
Then
\[
f\in M=\sum_{i=1}^{r}Ke_{i}\subset k.
\]
Replacing $x$ by $x^{1/s}$ for some $s$ in the case 2M, we may
assume that all the $e_{i}$ are in $\mathbb{C}((x)).$ The column
vector $\underline{e}=\,^{t}(e_{1},\dots,e_{r})$ satisfies
\[
\underline{e}(\sigma x)=B\underline{e}(x)
\]
for an invertible constant matrix $B\in GL_{r}(\mathbb{C})$. (In
the notation introduced above, $B=(^{t}A_{\sigma}^{0})^{-1}$). Write
$\underline{e}=\sum_{n=n_{0}}^{\infty}v_{n}x^{n}$ with $v_{n}\in\mathbb{C}^{r}.$
In case 2M this gives
\[
\sum_{n=n_{0}}^{\infty}v_{n}x^{pn}=\sum_{n=n_{0}}^{\infty}Bv_{n}x^{n},
\]
from where we deduce that $v_{n}=0$ for $n\ne0$, so each $e_{i}\in\mathbb{C}.$
In case 2Q the same equation gives
\[
\sum_{n=n_{0}}^{\infty}v_{n}p^{n}x^{n}=\sum_{n=n_{0}}^{\infty}Bv_{n}x^{n},
\]
from where we deduce that $v_{n}=0$ for $n$ sufficiently large,
since the matrix $B$ can have only finitely many eigenvalues. Thus
in this case, too, all the $e_{i}\in K$. As the $e_{i}$ are linearly
independent over $K$, in both cases we must have $r=1$ and $f\in Ke_{1}=K\subset k$.
\end{proof}

\section{The structure of formal $\Gamma$-difference modules\label{sec:Formal}}

The results of this part appear in various variations in the literature,
sometimes over the fields of Hahn series or of convergent power series
replacing the fields of Puiseux or Laurent series. The ideas date
back to works of Manin and Dieudonné on formal groups. We prove all
that we shall need later on in the global theory from first principles.
The reader may consult \cite{Roq,Sau,vdP-Re} and the references therein
for the historical development of the subject, and for further results. 

\subsection{Formal $(p,q)$-difference modules}

\subsubsection{Rank 1 formal $q$-difference modules}

In this section we prove an analogue of the Main Theorem over $k$
instead of $K,$ in the case 2Q of two $q$-difference operators.
The case of two Mahler operators will be discussed in the next section.
One starts by examining the structure of a $\Gamma$-difference module
$M$, for $\Gamma=\left\langle \tau\right\rangle $ infinite cyclic.
Adding a second multiplicatively independent and commuting operator
$\sigma$ imposes a serious restriction on the structure of $M$,
and forces it to descend to $\mathbb{C}$.

Let $q\in\mathbb{C}^{\times}$ and assume that $q$ is not a root
of unity. Fix once and for all a compatible sequence of roots $q^{1/s}.$
Let $\Gamma=\left\langle \tau\right\rangle $ act on the field $\widetilde{k}=\bigcup_{s=1}^{\infty}\mathbb{C}((x^{1/s}))$
via $\tau f(x^{1/s})=f(q^{1/s}x^{1/s}).$ (To get the results below
we have to work over $\widetilde{k}$, although $\tau$ is already
an automorphism of $k.$) Write $k_{s}=\mathbb{C}((x^{1/s})).$

A $\Gamma$-difference module over $\widetilde{k}$ is the same as
a $\widetilde{k}\left\langle \Phi,\Phi^{-1}\right\rangle $-module
which is finite dimensional over $\widetilde{k}$. Here the twisted
Laurent polynomials ring $\widetilde{k}\left\langle \Phi,\Phi^{-1}\right\rangle $
satisfies the relation
\[
\Phi a=\tau(a)\Phi
\]
for $a\in\widetilde{k}.$ We shall call a $\Gamma$-difference module
over $\widetilde{k}$ or over $k$ also a (formal) $q$-difference
module.

Rank-1 $q$-difference modules over $\widetilde{k}$ are classified
by $\widetilde{k}^{\times}/(\widetilde{k}^{\times})^{\tau-1},$ see
(i) from example \ref{exa:examples} with $r=1.$ Every element $f\in\widetilde{k}^{\times}$
can be written uniquely as $cx^{\lambda}f_{1}$ where $c\in\mathbb{C}^{\times}$,
$f_{1}\in U_{1}(k_{s})=1+x^{1/s}\mathbb{C}[[x^{1/s}]]$ (the principal
units of $k_{s})$ for some $s\in\mathbb{N}$, and $\lambda\in\mathbb{Q}.$
Since $f_{1}=g_{1}^{\tau-1}$ for $g_{1}\in U_{1}(k_{s})$ (solve
successively for the coefficients of $g_{1},$ using the fact that
$q$ is not a root of unity), and since $(x^{\mu})^{\tau-1}=q^{\mu},$
we see that classes in $\widetilde{k}^{\times}/(\widetilde{k}^{\times})^{\tau-1}$
are represented by $cx^{\lambda},$ where $\lambda$ (the \emph{slope})
is uniquely determined, and $c\in\mathbb{C}^{\times}$ is determined
up to multiplication by $q^{\alpha}$ for some $\alpha\in\mathbb{Q}.$
We therefore have the following easy Proposition.
\begin{prop}
Let $c\in\mathbb{C}^{\times}$ and $\lambda\in\mathbb{Q}$. Let $I_{\lambda,c}=\widetilde{k}e$
with $\Phi(e)=cx^{\lambda}e.$ Then every rank-1 $q$-difference module
over $\widetilde{k}$ is isomorphic to some $I_{\lambda,c}$ and $I_{\lambda,c}\simeq I_{\mu,d}$
if and only if $\lambda=\mu$ and $cd^{-1}=q^{\alpha}$ for some $\alpha\in\mathbb{Q}.$
\end{prop}

\subsubsection{Newton polygons}

We review well-known facts about Newton polygons and slopes. Let $\nu:\widetilde{k}^{\times}\to\mathbb{Q}$
be the valuation of $\widetilde{k}$, normalized by $\nu(x)=1.$ If
$P\in\widetilde{k}[T],$
\[
P(T)=a_{0}T^{r}+\cdots+a_{r-1}T+a_{r}
\]
$a_{0},a_{r}\ne0,$ we consider the points $(i,\nu(a_{i}))\in\mathbb{R}^{2}$
($0\le i\le r).$ The highest piecewise linear convex graph lying
on or below these points is called the Newton polygon $\mathcal{N}_{P}$
of $P$. It has two vertical edges, connecting $(0,\infty)$ to $(0,\nu(a_{0})),$
and $(r,\nu(a_{r}))$ to $(r,\infty).$ The other edges have rational
slopes $\lambda_{1}<\lambda_{2}<\cdots<\lambda_{s}$ and integral
horizontal lengths $r_{1},r_{2},\dots,r_{s}$ with $\sum r_{i}=r.$
The polynomial $P$ has precisely $r_{i}$ roots $\alpha$ in $\widetilde{k}^{\times}$
of valuation $\nu(\alpha)=\lambda_{i}.$ If we make a change of variable
$Q(T)=P(a^{-1}T)$ then $\mathcal{N}_{Q}$ has slopes $\lambda_{i}+\nu(a).$
After such a change of variables, we may therefore assume, when dealing
with Newton polygons, that the smallest slope of $P$ is 0. The definition
of $\mathcal{N}_{P}$ may be extended to an arbitrary non-zero $P\in\widetilde{k}[T,T^{-1}]$
so that $\mathcal{N}_{PT}$ is obtained from $\mathcal{N}_{P}$ by
a horizontal shift one unit to the right. It has the same slopes and
the same horizontal lengths.

If $P\in\widetilde{k}[T]$ is written as above we let $P(\Phi)=\sum_{i=0}^{r}a_{i}\Phi^{r-i}\in\widetilde{k}\left\langle \Phi,\Phi^{-1}\right\rangle .$
Note however that $P(T)\mapsto P(\Phi)$ is not a homomorphism, as
$\Phi$ does not commute with $\widetilde{k}$.

\bigskip{}

Let $M$ be a cyclic $q$-difference module, generated by the vector
$v.$ We shall later see (Birkhoff's cyclicity lemma) that \emph{every
}$q$-difference module is cyclic, but at this stage we do not know
it yet. Let $P(T)$ be a monic polynomial of minimal degree such that
$P(\Phi)v=0.$ Such a polynomial exists since the $\Phi^{i}v$ are
linearly dependent over $\widetilde{k}.$ Write
\begin{equation}
P(T)=T^{r}+a_{1}T^{r-1}+\cdots+a_{r-1}T+a_{r}.\label{eq:min.pol.}
\end{equation}
Then $a_{r}\ne0,$ since otherwise, as $\Phi^{-1}P(\Phi)v=0,$ the
polynomial $Q=\tau^{-1}(P)/T$ has degree $r-1$ and still satisfies
$Q(\Phi)v=0.$ Let $\mathcal{D}=\widetilde{k}\left\langle \Phi,\Phi^{-1}\right\rangle $
and consider the left ideal $\mathcal{D}P(\Phi).$ The homomorphism
of $\mathcal{D}$-modules
\[
\mathcal{D}/\mathcal{D}P(\Phi)\twoheadrightarrow M,
\]
sending $Q(\Phi)\in\mathcal{D}$ to $Q(\Phi)v$ is surjective. As
the module on the left is generated over $\widetilde{k}$ by $1,\Phi,\dots,\Phi^{r-1}$
and $M$ contains the linearly independent vectors $v,\Phi v,\dots,\Phi^{r-1}v,$
both sides have dimension $r$ and this map is an isomorphism.

If we replace $M$ by $M\otimes I_{\lambda,1}$ ($\lambda\in\mathbb{Q})$
and the cyclic vector $v$ by $v\otimes e$ where $e$ is the basis
of $I_{\lambda,1}$, then $\Phi^{i}(v\otimes e)=q^{\lambda\binom{i}{2}}x^{i\lambda}\Phi^{i}(v)\otimes e$
so the polynomial $P$ is replaced (up to a scalar multiple) by
\[
Q(T)=q^{-\lambda\binom{r}{2}}T^{r}+q^{-\lambda\binom{r-1}{2}}x^{\lambda}a_{1}T^{r-1}+\cdots+x^{(r-1)\lambda}a_{r-1}T+x^{r\lambda}a_{r}
\]
and the points $(i,\nu(a_{i}))$ by $(i,\nu(a_{i})+i\lambda).$ After
such a twist of $M$ we may assume that the slopes of $\mathcal{N}_{P}$
are $\ge0$ and that the first (smallest) slope is $0$.

Replacing the variable $x$ by some $x^{1/s}$ we may therefore assume
that all the roots of $P$ are in $k=\mathbb{C}((x))$, that the slopes
are integral, and that the smallest slope is 0. In particular, all
the $a_{i}\in\mathbb{C}[[x]]$.

\subsubsection{Factorization in $k\left\langle \Phi\right\rangle $}
\begin{lem}
\label{lem:factorization-1}Assume that in $(\ref{eq:min.pol.})$
the $a_{i}\in\mathbb{C}[[x]]$ $(i\ge1)$, and at least one of them
is a unit (these conditions are equivalent to the smallest slope of
$P$ being 0). Then there exists a unit $b_{0}\in\mathbb{C}[[x]]^{\times}$,$b_{1},\dots,b_{r-1}\in\mathbb{C}[[x]]$
and $c\in\mathbb{C}^{\times}$ such that in $k\left\langle \Phi\right\rangle $
we have
\[
P(\Phi)=\Phi^{r}+a_{1}\Phi^{r-1}+\cdots+a_{r}=\tau(b_{0})^{-1}(\Phi-c)(b_{0}\Phi^{r-1}+b_{1}\Phi^{r-2}+\cdots+b_{r-1}).
\]
\end{lem}

\begin{proof}
We write $a_{0}=1,$ $u=b_{0}$ and solve successively for the coefficient
of $\Phi^{r-i}$. Define
\[
b_{i}=\sum_{j=0}^{i}c^{j}\tau^{-j}(u)\tau^{-j-1}(a_{i-j})
\]
($0\le i\le r-1$), where $u\in\mathbb{C}[[x]]^{\times}$ and $c\in\mathbb{C}^{\times}$
are still to be determined. These equations guarantee that
\[
\tau(u)^{-1}(\tau(b_{i+1})-cb_{i})=a_{i+1}
\]
for $0\le i\le r-2$. To get the last coefficient we need
\[
-cb_{r-1}=\tau(u)a_{r}
\]
or, with $\tau(u)=v$, all that remains is to find $v\in\mathbb{C}[[x]]^{\times}$
and $c\in\mathbb{C}^{\times}$ such that
\[
\sum_{j=0}^{r}c^{j}\tau^{-j}(v)\tau^{-j}(a_{r-j})=0.
\]
Write $v=1+t_{1}x+t_{2}x^{2}+\cdots$ and $\tau^{-j}(a_{r-j})=a_{j,0}+a_{j,1}x+a_{j,2}x^{2}+\cdots.$
For $c$ we take a non-zero solution of
\[
\sum_{j=0}^{r}c^{j}a_{j,0}=0.
\]
Here we use the fact that since the smallest slope is 0, there is
$j<r$ with $a_{j,0}\ne0,$ and of course $a_{r,0}=1.$ We also insist
that for $i\ge1$ $q^{-i}c$ is not a root of the same polynomial.
This can be achieved because $q$ is not a root of unity, so we may
replace $c$ by the last element in the sequence $c,q^{-1}c,q^{-2}c,\dots$
solving the equation. 

We then solve successively for the $t_{i}$. We get
\[
t_{i}\left(\sum_{j=0}^{r}c^{j}q^{-ji}a_{j,0}\right)=r_{i}
\]
where $r_{i}$ is an expression involving the $a_{j,m}$, $c$ and
$t_{\ell}$ for $\ell<i$. By our assumption on $c$ the term in paranthesis
does not vanish, so we can solve for $t_{i}$.
\end{proof}
\begin{cor}
Assume that, after replacing $x$ by $x^{1/s}$ for some $s,$ the
smallest slope of $P$ is an integer $\mu.$ Then
\[
P(\Phi)=\tau(u)^{-1}(\Phi-cx^{\mu})uP_{1}(\Phi)
\]
where $P_{1}$ is a monic polynomial of degree $r-1$ and $u\in\mathbb{C}[[x]]^{\times}.$
\end{cor}

\begin{proof}
Let $\mu$ be the smallest slope of $P$ and consider the module $M\otimes I_{-\mu,1}$
with the cyclic vector $v\otimes e.$ Since
\[
(x^{\mu}\Phi)^{i}(v\otimes e)=\Phi^{i}v\otimes e
\]
we deduce that if $Q(T)=\sum_{i=0}^{r}a_{i}T^{i}$ is the monic minimal
polynomial of $v\otimes e$ then
\[
Q(\Phi)=q^{-\mu\binom{r}{2}}x^{-\mu r}P(x^{\mu}\Phi)
\]
(caution: it is not true that $Q(T)=q^{-\mu\binom{r}{2}}x^{-\mu r}P(x^{\mu}T)$;
the variable $T$ commutes with $x^{\mu}$ while $\Phi$ does not!).
The polynomial $Q$ has smallest slope 0, so by the Lemma
\[
q^{-\mu\binom{r}{2}}x^{-\mu r}P(x^{\mu}\Phi)=\tau(b_{0})^{-1}(\Phi-c)Q_{1}(\Phi)
\]
where $Q_{1}(\Phi)=\sum_{i=0}^{r-1}b_{i}\Phi^{i}$. Consider the automorphism
of the non-commutative ring $k\left\langle \Phi\right\rangle $ carrying
$\Phi$ to $x^{-\mu}\Phi$ and leaving $k$ fixed. (Note that it is
\emph{not }obtained by substituting $\Phi$ in a similar automorphism
of $k[T].$) Applying it to the above identity we get
\[
P(\Phi)=q^{\mu\binom{r}{2}}x^{\mu r}\tau(b_{0})^{-1}(x^{-\mu}\Phi-c)Q_{1}(x^{-\mu}\Phi)
\]
\[
=\tau(b_{0})^{-1}q^{\mu\binom{r}{2}-\mu(r-1)}(\Phi-cq^{\mu(r-1)}x^{\mu})x^{\mu(r-1)}Q_{1}(x^{-\mu}\Phi)
\]
\[
=\tau(b_{0})^{-1}(\Phi-cq^{\mu(r-1)}x^{\mu})q^{\mu\binom{r-1}{2}}x^{\mu(r-1)}Q_{1}(x^{-\mu}\Phi)
\]
\[
=\tau(b_{0})^{-1}(\Phi-cq^{\mu(r-1)}x^{\mu})P_{1}(\Phi)
\]
where the leading coefficient of $P_{1}$ is $b_{0}.$ The claim follows,
with $u=b_{0}$ and $c$ replaced by $cq^{\mu(r-1)}.$ Note that $c$
is anyhow only determined by $M$ up to a power of $q$, since $I_{\lambda,c}\simeq I_{\lambda,qc}$.
\end{proof}
Consider the vector $e_{1}=uP_{1}(\Phi)v\ne0.$ Then
\[
\Phi e_{1}=cx^{\mu}e_{1}
\]
so $\widetilde{k}e_{1}\simeq I_{\mu,c}.$

It is easy to see that the slopes of $P$ are the slopes of $P_{1}$
and $\mu$ (with multiplicities).

\subsubsection{The structure theorem for a $q$-difference module over $\widetilde{k}$}
\begin{prop}
\label{prop:filtration_q_difference}Let $M$ be an arbitrary $q$-difference
module over $\widetilde{k}$. Then $M$ has an ascending filtration
with one-dimensional graded pieces of the form $I_{\mu_{i},c_{i}}$
with rational slopes $\mu_{1}\le\mu_{2}\le\cdots\le\mu_{r}.$
\end{prop}

\begin{proof}
It is enough to prove that $M$ contains a rank 1 submodule $M_{1}$,
because then we continue by induction on $M/M_{1}.$ For that we may
assume that $M$ is cyclic, and the claim follows from what was done
above.
\end{proof}
Since the Jordan-Hölder factors of $M$ are intrinsic to $M$ we deduce
that if $M$ is cyclic the slopes are independent of the cyclic vector
used in the proof.
\begin{cor}[Birkhoff's cyclic vector lemma ]
 Every $q$-difference module over $\widetilde{k}$ has a cyclic
vector. 
\end{cor}

\begin{proof}
We prove the corollary by induction on the rank, the rank 1 case being
obvious. Let $M_{1}\subset M$ be a submodule of rank $1$ and $v\in M$
a vector projecting to a cyclic vector of $M/M_{1}.$ Let $e$ be
a basis element of $M_{1}.$ Let $P(\Phi)$ be a polynomial in $\Phi$
with coefficients in $\widetilde{k}$ annihilating $v$. For an appropriate
$\lambda$, $P(\Phi)(x^{\lambda}e)\ne0.$ Replacing $e$ by $x^{\lambda}e$
we may assume that $P(\Phi)e\ne0.$ But then $u=v+e$ is a cyclic
vector for $M,$ as the module generated by it contains $P(\Phi)u=P(\Phi)e,$
hence $M_{1},$ and modulo $M_{1}$ it contains the image of $v,$
hence projects onto $M/M_{1}.$
\end{proof}
\begin{thm}[Structure theorem for formal $q$-difference modules]
\label{thm:Structure}  Let $M$ be a $q$-difference module over
$\widetilde{k}.$ Let $\lambda_{1}<\lambda_{2}<\cdots<\lambda_{m}$
be the distinct slopes of $M$ in increasing order. Then there are
$\mathbb{C}$-vector spaces $N_{i}\subset M$ with endomorphisms $\phi_{i}\in GL_{\mathbb{C}}(N_{i})$
so that
\[
M=\bigoplus_{i=1}^{m}\widetilde{k}\otimes_{\mathbb{C}}N_{i}
\]
and $\Phi(1\otimes v_{i})=x^{\lambda_{i}}\otimes\phi_{i}(v_{i})$
for $v_{i}\in N_{i}.$ If $M$ is defined over $k$ then the same
is true if we extend scalars to $k_{s}$ where $s$ is the least common
denominator of the $\lambda_{i}$.
\end{thm}

\begin{proof}
We may assume that $M$ is defined over $k$ and that all the slopes
are integral. If this is not the case, simply replace the variable
$x$ by $x^{1/s}$. In view of the last corollary we may assume that
$M$ is generated by a cyclic vector $v$ and we let $P(T)$ be the
unique monic polynomial of degree $r=rk(M)$ such that $P(\Phi)v=0.$

We shall prove the theorem in two stages. First, we show that there
exists a basis of $M$ with respect to which $\Phi$ is represented
by a matrix $A=(a_{ij})\in GL_{r}(k)$ where $a_{ii}=c_{i}x^{\mu_{i}}$
($c_{i}\in\mathbb{C}^{\times}$, $\mu_{1}\le\mu_{2}\le\cdots\le\mu_{r}$
are the $\lambda_{i}$ with multiplicities), and $a_{ij}=0$ unless
$j=i$ or $j=i+1$. In particular, $A$ is upper triangular.

Indeed, using Lemma \ref{lem:factorization-1} and its Corollary repeatedly
we may write
\[
P(\Phi)=u_{0}(\Phi-c_{1}x^{\mu_{1}})u_{1}(\Phi-c_{2}x^{\mu_{2}})u_{2}\cdots u_{r-1}(\Phi-c_{r}x^{\mu_{r}})u_{r}
\]
with $u_{i}\in\mathbb{C}[[x]]^{\times}.$ Let
\[
e_{i}=\left\{ u_{i}(\Phi-c_{i+1}x^{\mu_{i+1}})u_{i+1}\cdots u_{r-1}(\Phi-c_{r}x^{\mu_{r}})u_{r}\right\} v.
\]
Since $v,\Phi v,\dots,\Phi^{r-1}v$ is a $k$-basis of $M,$ so is
$e_{1},e_{2},\dots,e_{r}.$ For $1\le i\le r$ 
\[
(\Phi-c_{i}x^{\mu_{i}})e_{i}=u_{i-1}^{-1}e_{i-1}
\]
(where $e_{0}=0$) so the matrix of $\Phi$ in the basis $\{e_{i}\}$
satisfies $a_{ii}=c_{i}x^{\mu_{i}}$ and $a_{i-1,i}=u_{i-1}^{-1},$
while all the other $a_{ij}=0$.

We may assume, without loss of generality, that if $\mu_{i}=\mu_{j}$
then $c_{i}/c_{j}\notin q^{\mathbb{Z}},$ unless $c_{i}=c_{j}.$ Indeed,
if this were the case, and say $j>i$, replace $e_{j}$ by some $x^{n}e_{j},$
replacing $c_{j}$ by $c_{i}$. We may no longer be able to assume
that the $a_{i-1,i}$ are units, but we shall not be using this.

It is now enough to prove the following. Let $\mu_{1}\le\cdots\le\mu_{r}$
be integers. Let $c_{i}\in\mathbb{C}^{\times}$ be such that whenever
$\mu_{i}=\mu_{j}$, $c_{i}/c_{j}\notin q^{\mathbb{Z}}$, unless $c_{i}=c_{j}.$
Assume that $e_{1},\dots,e_{r}$ is a basis of $M$, $1<n\le r$ and
for all $j<n$
\[
\Phi(e_{j})=x^{\mu_{j}}\left(c_{j}e_{j}+\sum_{i=1}^{j-1}a_{ij}e_{i}\right)
\]
where $a_{ij}\in\mathbb{C}$ and $a_{ij}=0$ unless $(\mu_{i},c_{i})=(\mu_{j},c_{j})$.
Assume that
\[
\Phi(e_{n})=x^{\mu_{n}}\left(c_{n}e_{n}+\sum_{i=1}^{n-1}g_{i}e_{i}\right)
\]
($g_{i}\in k$). Then there exists an
\[
\tilde{e}_{n}=e_{n}+\sum_{i=1}^{n-1}h_{i}e_{i}
\]
such that
\[
\Phi(\tilde{e}_{n})=x^{\mu_{n}}\left(c_{n}\tilde{e}_{n}+\sum_{i=1}^{n-1}a_{in}e_{i}\right)
\]
with $a_{in}\in\mathbb{C}$ and $a_{in}=0$ unless $(\mu_{i},c_{i})=(\mu_{n},c_{n})$.
Using this inductively we modify the basis with which we started until
we get a basis w.r.t. which $\Phi$ has the form described in the
theorem.

Twisting $M$ by $I_{-\mu_{n},1}$ we may assume that $\mu_{n}=0.$

We consider the $h_{i}\in k$ and the $a_{in}$ as variables and solve
for them inductively, starting with $i=n-1$ and going down. Collecting
terms (including the terms arising from the $h_{j}$ for $j>i$) we
get that we have to solve
\[
c_{n}h_{i}-c_{i}\tau(h_{i})x^{\mu_{i}}=f_{i}-a_{in}
\]
for some $f_{i}\in k$. Recall that $\mu_{i}\le0=\mu_{n}.$ Now if
$\mu_{i}<0$ this has a solution $h_{i}$, with $a_{in}=0.$ If $\mu_{i}=0$
but $c_{i}\ne c_{n}$ then $c_{i}q^{m}\ne c_{n}$ for all $m$ by
assumption and again there is a solution with $a_{in}=0.$ Finally,
if $\mu_{i}=0$ and $c_{i}=c_{n}$ we can cancel out all the terms
of $f_{i}$ except for the constant one, which we kill with $a_{in}.$
This concludes the proof of the theorem.
\end{proof}
\begin{cor}
A $q$-difference module over $k$ (or $\widetilde{k}$) descends
to $\mathbb{C}$ if and only if all its slopes are 0.
\end{cor}

Borrowing terminology from differential equations, such a module is
also called \emph{regular-singular}. We shall not be using this terminology.
\begin{rem}
\label{rem:Resonants}The $\mathbb{C}$-subspaces $N_{i}$ are not
unique. In fact, $N_{i}$ can be replaced by $x^{\mu_{i}}N_{i}$ ($\mu_{i}\in\mathbb{Q})$
and $\phi_{i}$ by $q^{\mu_{i}}\phi_{i}.$ Two $\mathbb{C}$-subspaces
related in this way will be called \emph{resonants} of each other.
It can be checked that this is the only source of non-uniqueness in
Theorem \ref{thm:Structure}.
\end{rem}

\subsubsection{The structure theorem for a $(p,q)$-difference module over k}

We now introduce a second operator $\sigma f(x)=f(px)$ for $p\in\mathbb{C}^{\times}$
such that $p$ and $q$ are multiplicatively independent. We let $\Gamma=\left\langle \sigma,\tau\right\rangle \subset Aut(k)$
and call a $\Gamma$-difference module over $k$ a (formal) $(p,q)$-difference
module. Such a module is clearly a $q$-difference module, and we
shall show that the introduction of the second operator $\Phi_{\sigma}$,
commuting with $\Phi_{\tau}$, imposes serious restrictions on its
structure, and forces it to descend to $\mathbb{C}.$
\begin{thm}
\label{thm:2Q formal structure} Let $M$ be a $(p,q)$-difference
module over $k,$ for multiplicatively independent $p$ and $q$.
Then $M$ descends to $\mathbb{C}$.
\end{thm}

\begin{proof}
Consider the extension of scalars $M_{\widetilde{k}}$ and a decomposition
\[
M_{\widetilde{k}}=\bigoplus_{i=1}^{m}\widetilde{k}\otimes_{\mathbb{C}}N_{i},
\]
as given by Theorem \ref{thm:Structure}. Let $\lambda$ be a slope
of $M_{\widetilde{k}}$. Then there exists a vector $v$ with $\Phi_{\tau}(v)=cx^{\lambda}v$,
for some $c\in\mathbb{C^{\times}}$. Applying $\Phi_{\sigma}$ we
have
\[
\Phi_{\tau}(\Phi_{\sigma}v)=\Phi_{\sigma}(\Phi_{\tau}v)=\Phi_{\sigma}(cx^{\lambda}v)=cp^{\text{\ensuremath{\lambda}}}x^{\lambda}\Phi_{\sigma}(v).
\]
That is, $\Phi_{\sigma}(v)$ is an eigenvector with eigenvalue $cp^{\lambda}x^{\lambda}$.
Iterating we find that $I_{\lambda,cp^{n\lambda}}$ appears as a Jordan-Hölder
constituent of $M$ for all $n\in\mathbb{\mathbb{\mathbb{\mathbb{N}}}}$.
Since there are finitely many Jordan-Hölder factors, we have that
for some $n>m\ge0$, $I_{\lambda,cp^{m\lambda}}\simeq I_{\lambda,cp^{n\lambda}}$.
Thus for some $\alpha\in\mathbb{Q}$ we must have $p^{(n-m)\lambda}q^{\text{\ensuremath{\alpha}}}=1$.
It follows that $\lambda=\alpha=0$, since $p$ and $q$ are multiplicatively
independent, i.e. the only possible $q$-slope is 0.

Theorem \ref{thm:Structure} implies that there is a basis $e_{1},\ldots,e_{r}$
of $M$ over $k$ such that $[\Phi_{\tau}]_{e}=A_{\tau}^{-1}$, with
$A_{\tau}\in GL_{r}(\mathbb{C})$. Furthermore, the basis can be chosen
so that if $c$ and $c'$ are two distinct eigenvalues of $A_{\tau}$,
$c/c'\notin q^{\mathbb{Z}}.$ Let $A_{\sigma}^{-1}$ be the matrix
of $\Phi_{\sigma}$ in the same basis. We claim that it is also a
constant matrix. In fact, if we write $A_{\sigma}(x)=\sum B_{n}x^{n}$
with $B_{n}\in M_{r}(\mathbb{C})$, the consistency equation 
\[
A_{\tau}A_{\sigma}(x)=A_{\sigma}(qx)A_{\tau}
\]
implies that
\[
A_{\tau}B_{n}A_{\tau}^{-1}=q^{n}B_{n}.
\]
The eigenvalues of $A_{\tau}$ in its adjoint action are all of the
form $c/c'$ for eigenvalues $c,c'$ of $A_{\tau}.$ By our assumption
$c/c'=q^{n}$ only if $n=0.$ Hence $B_{n}=0$ unless $n=0,$ and
$A_{\sigma}$ is a constant matrix. We conclude that 
\[
M_{0}=\oplus_{i=1}^{r}\mathbb{C}e_{i}
\]
is an underyling $\mathbb{C}$-structure of $M$.
\end{proof}
As in theorem \ref{thm:Structure}, the $\mathbb{C}$-structure $M_{0}$
is not unique, because of the existence of resonants.

\subsection{Formal $(p,q)$-Mahler modules}

\subsubsection{Rank-1 formal $q$-Mahler modules}

We turn our attention to the case of formal Mahler modules. These
are $\Gamma$-difference modules over the field of Puiseux series
$\widetilde{k},$ where $\Gamma=\left\langle \tau\right\rangle $
or $\Gamma=\left\langle \sigma,\tau\right\rangle $, and $\sigma f(x)=f(x^{p})$,
$\tau f(x)=f(x^{q})$ are Mahler operators for $p,q\in\mathbb{N}$
multiplicatively independent. We shall call such $\Gamma$-difference
modules (formal) $q$-Mahler modules, when $\Gamma=\left\langle \tau\right\rangle ,$
and (formal) $(p,q)$-Mahler modules when $\Gamma=\left\langle \sigma,\tau\right\rangle .$
As before, we start by studying the structure of $q$-Mahler modules,
and then examine the restriction imposed by the introduction of the
second operator $\Phi_{\sigma}.$

Rank-1 $q$-Mahler modules are classified by $\widetilde{k}^{\times}/(\widetilde{k}^{\times})^{\tau-1}.$
Once again, we pick $f\in\widetilde{k}^{\times}$ and write $f=cx^{\lambda}f_{1}$
where $c\in\mathbb{C}^{\times}$ and $f_{1}\in U_{1}(k_{s}).$ This
time both $f_{1}$ and $x^{\lambda}$ are in $(\widetilde{k}^{\times})^{\tau-1},$
as $(x^{\mu})^{\tau-1}=x^{\mu(q-1)}.$ We get the following easy Proposition.
\begin{prop}
Every rank-1 $q$-Mahler module over $\widetilde{k}$ descends to
$\mathbb{C}.$ Let $I_{c},$ for $c\in\mathbb{C}^{\times},$ be the
rank-1 module $\widetilde{k}e,$ with $\Phi_{\tau}(e)=ce.$ Then every
rank-1 module is isomorphic to $I_{c}$ for a unique $c$.
\end{prop}

\subsubsection{Factorization in $\widetilde{k}\left\langle \Phi\right\rangle $}

We consider the twisted polynomial ring $\widetilde{k}\left\langle \Phi\right\rangle $
consisting of polynomials
\[
\sum_{i=0}^{n}a_{i}\Phi^{n-i},
\]
$(a_{i}\in\widetilde{k}$) where $\Phi a=\tau(a)\Phi.$ 
\begin{lem}
\label{lem:Factorization}\emph{ }Let $\sum_{i=0}^{n}a_{i}\Phi^{n-i}\in\widetilde{k}\left\langle \Phi\right\rangle $,
and assume $a_{0}=1,\,\,a_{n}\neq0.$ Then there exist $c\in\mathbb{C}^{\times},\,\mu\in\mathbb{Q},\,b_{0},\dots,b_{n-1}\in\widetilde{k}$
such that $b_{0}\in U_{1}(\widetilde{k})$, $b_{n-1}\neq0$ and
\[
\sum_{i=0}^{n}a_{i}\Phi^{n-i}=\tau(b_{0})^{-1}(\Phi-cx^{\mu})\sum_{i=0}^{n-1}b_{i}\Phi^{n-1-i}.
\]
\end{lem}

{[}Compare Chapter IV, §4 Lemma 2 in Demazure's \emph{Lectures on
$p$-divisible groups }LNM 302 (1972) Springer-Verlag. That lemma
is key to the Manin-Dieudonné classification of $F$-isocrystals over
an algebraically closed field of characteristic $p,$ or - what amounts
to the same - the classification of $p$-divisible groups over such
a field up to isogeny.{]}
\begin{proof}
To simplify the notation we write, in the proof of the lemma only,
$a^{(i)}=\tau^{i}(a).$ Write also $u=b_{0}^{(1)}.$ We have to find
$\mu,u,b_{1},\dots,b_{n-1}$ and $c$ as in the lemma satisfying the
equations
\begin{equation}
ua_{i}=b_{i}^{(1)}-cx^{\mu}b_{i-1}\,\,\,\,\,\,(0\le i\le n)\label{eq:a and b}
\end{equation}
($b_{-1}=b_{n}=0)$. Solving successively for $b_{i}$ we get the
equation
\begin{equation}
(ua_{0})c^{n}+(u^{(1)}a_{1}^{(1)}x^{-\mu q})c^{n-1}+\cdots+(u^{(n)}a_{n}^{(n)}x^{-\mu(q+\cdots+q^{n})})=0,\label{eq:c}
\end{equation}
which we have to solve for $u\in U_{1}(\widetilde{k})$ and $c\in\mathbb{C}^{\times}.$
Let
\[
\mu=\min_{1\le i\le n}\left(\frac{1-1/q}{1-1/q^{i}}\right)\nu(a_{i})
\]
where $\nu$ is the valuation on $\widetilde{k}$, normalized by $\nu(x)=1.$
Note that
\[
\nu(a_{i}^{(i)}x^{-\mu(q+\cdots+q^{i})})=q^{i}\left(\nu(a_{i})-\mu\left(\frac{1-1/q^{i}}{1-1/q}\right)\right)\ge0
\]
and there exists an index $i\ge1$ for which this is 0. This means
that the expression $a_{i}^{(i)}x^{-\mu(q+\cdots+q^{i})}$, appearing
together with $u^{(i)}$ as the coefficient of $c^{n-i}$, is integral,
i.e. has no pole, and at least one such expression, besides the leading
one, is a unit. Replacing $x$ by $x^{s}$ for a suitable $s$, we
may assume that all the exponents of $x$ appearing in (\ref{eq:c})
are integral. We solve (\ref{eq:c}) modulo higher and higher powers
of $x,$ setting
\[
u=1+d_{1}x+d_{2}x^{2}+\cdots
\]
and choosing the $d_{m}$ successively. By what we have seen, there
exists a $c\neq0$ in $\mathbb{C}$ solving (\ref{eq:c}) modulo $x$
(i.e. substituting $x=0).$ Noting that
\[
u^{(i)}=1+d_{1}x^{q^{i}}+d_{2}x^{2q^{i}}+\cdots
\]
 it is then an easy matter to solve successively for the $d_{m}.$
\end{proof}
\begin{cor}
(Compare with Theorem 15 in \cite{Roq}.) Every monic polynomial from
$\widetilde{k}\left\langle \Phi\right\rangle $ factors as
\end{cor}

\[
u_{0}(\Phi-c_{1}x^{\mu_{1}})u_{1}(\Phi-c_{2}x^{\mu_{2}})u_{2}\cdots u_{n-1}(\Phi-c_{n}x^{\mu_{n}})u_{n}
\]
where the $\mu_{j}\in\mathbb{Q},$ $\mu_{j}\le q\mu_{j+1},$ $c_{j}\in\mathbb{C}^{\times}$
and $u_{j}\in U_{1}(\widetilde{k})$.
\begin{proof}
Apply the lemma inductively. The relation $\mu_{j}\le q\mu_{j+1}$
follows from the inequality
\[
q\nu(b_{i})\ge\mu(1+\frac{1}{q}+\cdots+\frac{1}{q^{i-1}})
\]
which is proved by induction on $i,$ based on (\ref{eq:a and b}).
\end{proof}

\subsubsection{The structure theorem for a $q$-Mahler module over $\widetilde{k}$}

Let $M$ be a $q$-Mahler module over $\widetilde{k}.$ Similarly
to Proposition \ref{prop:filtration_q_difference} we get the following
structure theorem for $M.$
\begin{thm}
(Compare with Theorem 9 in \cite{Roq}.) \label{thm:Formal q-Mahler filtration}Let
$M$ be a $q$-Mahler module over $\widetilde{k}.$ Then $M$ has
an ascending filtration with one-dimensional graded pieces of the
form $I_{c}$.
\end{thm}

\begin{proof}
It is enough to prove that any $q$-Mahler module over $\widetilde{k}$
has a rank 1 submodule. Let $u$ be any non-zero vector in $M,$ and
$n$ the minimal number such that $u,\Phi u,\dots,\Phi^{n}u$ are
linearly dependent over $\widehat{K}$. Let $\sum_{i=0}^{n}a_{i}\Phi^{n-i}u=0$
be a linear dependence with $a_{0}=1,$ and decompose the polynomial
as in the lemma. Let $v=\sum_{i=0}^{n-1}b_{i}\Phi^{n-1-i}u.$ Note
that $v\neq0$ by our assumption on $n.$ Then $\Phi v=cx^{\mu}v.$
If $\mu\neq0$ replace $v$ by $x^{-\mu/(q-1)}v$.
\end{proof}
Contrary to Theorem \ref{thm:Structure} we do not have at our disposal
a more refined structure theorem describing the off-diagonal entries
in the resulting upper triangular matrix associated with $\Phi.$
One can not expect to have all the entries in $\mathbb{C}$, because
a general $q$-Mahler module need not descend to $\mathbb{C}.$ However,
the theorem we have just proved suffices to obtain the Mahler analogue
of Theorem \ref{thm:2Q formal structure}.

\subsubsection{The structure theorem for a $(p,q)$-Mahler module over $\widetilde{k}$}

We now consider a \emph{pair} of operators $\sigma$ and $\tau$ as
above
\[
\sigma f(x)=f(x^{p}),\,\,\,\,\tau f(x)=f(x^{q})
\]
for multiplicatively independent $p,q\in\mathbb{N}.$ Let $\Gamma=\left\langle \sigma,\tau\right\rangle \subset Aut(\widetilde{k}).$
A $\Gamma$-difference module over $\widetilde{k}$ will be called
a (formal) $(p,q)$-Mahler module.
\begin{thm}
\label{lem:2M formal structure}Every $(p,q)$-Mahler module over
$\widetilde{k}$ admits a unique $\mathbb{C}$-structure.
\end{thm}

\begin{proof}
In terms of matrices, we have to show that any two matrices $A=A_{\tau}$
and $B=A_{\sigma}$ in $GL_{r}(\widetilde{k})$ satisfying the consistency
condition
\[
\sigma(A)B=\tau(B)A
\]
are gauge-equivalent to a pair of commuting constant matrices, unique
up to conjugation.

The uniqueness is easy. Suppose $(A,B)$ is a commuting pair of constant
matrices and $C\in GL_{r}(\widetilde{k})$ is such that
\[
(A',B')=(\tau(C)^{-1}AC,\sigma(C)^{-1}BC)
\]
are also constant. Replacing $x$ by some $x^{s}$ we may assume that
the entries of $C$ are all in $k$. Then 
\[
C=A^{-1}\tau(C)A'=A^{-2}\tau^{2}(C)A'^{2}=\cdots
\]
so $C\in GL_{r}(\mathbb{C})$, because its Laurent expansion is supported
in degrees divisible by $q^{n}$ for every $n.$

We next remark that if $(A,B)$ is a consistent pair in $GL_{r}(\widetilde{k})$
with $A\in GL_{r}(\mathbb{C})$ then $B\in GL_{r}(\mathbb{C})$ as
well. Indeed, if $A$ is constant the consistency equation takes the
form
\[
AB=\tau(B)A.
\]
Under a change of variable we may assume that the entries of $B$
are all in $k$. As above, this yields
\[
B=A^{-1}\tau(B)A=A{}^{-2}\tau^{2}(B)A^{2}=\cdots
\]
so $B\in GL_{r}(\mathbb{C})$.

Let $M$ be a $(p,q)$-Mahler module over $\widetilde{k}$. Theorem
\ref{thm:Formal q-Mahler filtration} guarantees that for some $c\in\mathbb{C}^{\times}$
the space
\[
W=\{v\in M|\,\Phi_{\tau}v=cv\}
\]
is non-zero. It is easily seen that vectors in $W$ which are linearly
independent over $\mathbb{C}$ are also linearly independent over
$\widetilde{k}.$ Indeed, if $\sum_{i=1}^{m}a_{i}v_{i}=0$ is a shortest
linear dependence over $\widetilde{k}$ between some $\mathbb{C}$-independent
vectors in $W$, with $a_{1}=1,$ apply $\Phi_{\tau}$ to get (after
dividing by $c$) $\sum_{i=1}^{m}\tau(a_{i})v_{i}=0$. This shows
that all $\tau(a_{i})=a_{i},$ hence $a_{i}\in\mathbb{C}$, or else
we get by subtraction a shorter linear dependence. But this contradicts
the linear independence of the $v_{i}$ over $\mathbb{C}$. It follows
that $W$ is finite dimensional over $\mathbb{C}.$ It is evidently
preserved by $\Phi_{\sigma}.$ Thus we may find an eigenvector $e\in W$
for $\Phi_{\sigma},$ namely $\Phi_{\sigma}(e)=de,$ for $d\in\mathbb{C}^{\times}.$
This means that $\widetilde{k}e=M_{1}$ is a rank-1 $(p,q)$-Mahler
submodule of $M.$ Continuing in this way with $M/M_{1}$ etc. we
arrive at a filtration of $M$ by $(p,q)$-Mahler submodules, whose
graded pieces are of rank 1 and admit a $\mathbb{C}$-structure. 

In terms of the matrices $A,B$ with which we started, this means
that we may assume that they are lower triangular, with diagonal entries
in $\mathbb{C}^{\times}.$ It remains to prove that they are gauge-equivalent
to a lower triangular pair $(A',B')$ with $A'$ constant. As mentioned
above, the fact that $B'$ is also constant will follow suit. Write

\[
A=\left(\begin{array}{cc}
A_{11} & 0\\
A_{21} & A_{22}
\end{array}\right),\;\;\;B=\left(\begin{array}{cc}
B_{11} & 0\\
B_{21} & B_{22}
\end{array}\right)
\]
with $A_{11}\in\mathbb{C}^{\times}$, $A_{21}\in M_{r-1,1}(\widetilde{k})$,
$A_{22}\in GL_{r-1,r-1}(\widetilde{k})$, and similarly for $B$.
The consistency equation for $A$ and $B$ implies the same equation
for $A_{22}$ and $B_{22}$. Hence by induction we may assume that
$A_{22}$ and $B_{22}$ are constant lower triangular. It remains
to descend the constants in $A_{21}$.

The consistency equation now takes the form
\begin{equation}
A_{21}(x^{p})B_{11}+A_{22}B_{21}(x)=B_{21}(x^{q})A_{11}+B_{22}A_{21}(x).\label{eq:A21consistency}
\end{equation}
After a change of variable we may assume that all the exponents appearing
in the equations are integers. We will show that if $A_{21}(x)$ or
$B_{21}(x)$ have a pole at 0, replacing the pair $(A,B)$ by an equivalent
pair, without affecting the diagonal blocks, we can reduce the order
of the pole, until we get rid of the polar parts altogether. To simplify
the argument we shall assume that $(p,q)=1.$ We shall explain how
to get rid of this assumption at the end of the proof.

Let $Mx^{-m}$ be the lowest term in $A_{21}(x)$ and $Nx^{-n}$ the
lowest term in $B_{21}(x)$ where $M,N\in M_{r-1,1}(\mathbb{C})$.
Assume that there is a pole, i.e. $n,m\ge1$ (otherwise there is nothing
to prove). Then looking at the lowest order terms in (\ref{eq:A21consistency})
gives $qn=pm$ and $MB_{11}=NA_{11}$. By our assumption that $(p,q)=1,$
$m/q=n/p$ is an integer. Let
\[
C(x)=\left(\begin{array}{cc}
I & 0\\
-MA_{11}^{-1}x^{-m/q} & I
\end{array}\right).
\]
Then the pair $(\widetilde{A},\widetilde{B})=(\tau(C)AC^{-1},\sigma(C)BC^{-1})$
has the same shape of $(A,B)$ with 
\begin{eqnarray*}
\widetilde{A}_{21}(x) & = & -Mx^{-m}+A_{21}(x)+A_{22}MA_{11}^{-1}x^{-m/q},
\end{eqnarray*}
\begin{eqnarray*}
\widetilde{B}_{21}(x) & = & -Nx^{-n}+B_{21}(x)+B_{22}NB_{11}^{-1}x^{-n/p}.
\end{eqnarray*}
The order of the poles of $\widetilde{A}$ and $\widetilde{B}$ is
smaller than their order in $A$ and $B$. Continuing inductively
we can eliminate the polar parts altogether.

We may therefore assume that the pair $(A,B)$ has no poles. To conclude
we need to solve the equation
\[
\widetilde{A}_{21}(x)=A_{21}(x)+C_{21}(x^{q})A_{11}-A_{22}C_{21}(x),
\]
for $C_{21}(x)$ so that $\widetilde{A}_{21}(x)$ is constant. Taking
the left hand side to be $A_{21}(0),$ we can find $C_{21}\in M_{r-1,1}(x\mathbb{C}[[x]])$
solving succesively for the coefficients of $x^{n}$, $n\geq1.$ This
concludes the proof under the assumption that $(p,q)=1.$

If $(p,q)=\ell>1$ the Laurent expansions of $\widetilde{A}_{21}$
and $\widetilde{B}_{21}$ which were constructed in the first step
might have a term with fractional degree, with denominator dividing
$\ell$ (the denominator in $m/q=n/p$). Ignore this issue and continue
inductively as before, each time removing the terms of lowest degrees.
As long as we have not reached the terms of degree $-m/q$ in $\widetilde{A}_{21}$,
the lowest terms in it will have integral degree $-m<-m'<-m/q$, and
we will be able to remove these lowest terms by a gauge transformation
as above, introducing a term with fractional degree with denominator
at worst $\ell,$ in degree $-m/q<-m'/q.$ Symmetrically, we may remove
all the polar part of $\widetilde{B}_{21}$ up to degree $-n/p$,
introducing at worst $\ell$ in the denominators of the exponents
of $x$. Once we reach the first fractional degree, we substitute
$\xi^{\ell}$ for $x,$ getting new power series with \emph{integral
degrees} in $\xi$ (i.e. matrix entries in $\mathbb{C}((\xi))$) satisfying
$(\ref{eq:A21consistency})).$ If $p=\ell p'$ and $q=\ell q'$ then
the lowest degree in the new $A_{21}(\xi)$ will be $-m/q'$ and the
lowest degree in the new $B_{21}$ will be $-n/p'.$ As at least one
of $p',q'$ is $>1$, we can continue by induction until we remove
all the polar part as before.
\end{proof}
\begin{rem}
A careful analysis of the proof of the theorem shows that the role
of the second Mahler operator $\Phi_{\sigma}$ in it was minor. It
was only used to guarantee that the process of reducing the matrix
$C$ to a constant matrix by means of gauge equivalence transformations
terminates after finitely many steps. If we replace the field of Puiseux
series by the field of Hahn series we can get rid of the polar parts
of the entries in $C$ in \emph{countably} many steps that yield a
\emph{convergent} Hahn series (with matrix coefficients). Once the
polar parts have been eliminated, the rest of the proof is the same.
Thus the above proof can be modified to prove the main theorem (Theorem
2) of \cite{Roq}, that any $q$-Mahler module over the field of Hahn
series descends to $\mathbb{C}.$ Moreover, if the original module
was defined over $\mathbb{C}((x))$ then the matrix $C$ needed to
descend its structure to $\mathbb{C}$ (i.e. to make $C(x^{q})^{-1}A(x)C(x)=A_{0}$
constant) would have entries in Hahn series whose supports are well-ordered
subsets of $\mathbb{Z}[1/q]$. In essence, this is the approach taken
by Julien Roques.

This remark should be contrasted with Theorem \ref{thm:2Q formal structure}.
In the case 2Q the second ($p$-difference) operator was used in a
more substantial way, to guarantee that the slopes of the first ($q$-difference)
operator were all $0$, and vice versa.
\end{rem}

\section{The structure of rational $\Gamma$-difference modules\label{sec:Rational}}

In this part we follow the method of \cite{Sch-Si1,Sch-Si2}.

\subsection{$(p,q)$-Mahler modules}

In this section we prove Theorem \ref{thm:Main Theorem} in the case
2M. Recall that
\[
\widetilde{K}=\bigcup_{s=1}^{\infty}\mathbb{C}(x^{1/s}),\,\,\,\widetilde{k}=\bigcup_{s=1}^{\infty}\mathbb{C}((x^{1/s})),
\]
and the two Mahler operators are
\[
\sigma(x^{1/s})=x^{p/s},\,\,\,\tau(x^{1/s})=x^{q/s}
\]
where $p$ and $q$ are multiplicatively independent natural numbers. 

Let $\Gamma=\left\langle \sigma,\tau\right\rangle \subset Aut(\widetilde{K})$
(or $Aut(\widetilde{k})$). Then $\Gamma$ is free abelian of rank
$2$. Let $M$ be a rank $r$ $\Gamma$-difference module over $\widetilde{K}$
(called also a $(p,q)$-Mahler module). Fix a basis of $M$ and let
$A=A_{\sigma}$ and $B=A_{\tau}$ be the matrices attached to $\Phi_{\sigma}$
and $\Phi_{\tau}$ in this basis as in \S\ref{subsec:Matrices}.
Changing variables, and writing $x$ for $x^{1/s}$ if necessary,
we may assume that $A,B\in GL_{r}(K),$ where $K=\mathbb{C}(x).$

Let $i=0,1$ or $\infty.$ Let $t_{0}=x,$ $t_{1}=x-1$ and $t_{\infty}=1/x$
be local parameters at the corresponding point. At the point $1$
we shall also use $z=\log(x)=\log(1+t_{1})$ as a formal parameter.
Let $k_{i}=\mathbb{C}((t_{i}))$ be the completion of $K=\mathbb{C}(x)$
at the point $i$. By base-change we may regard $M_{\widetilde{k}_{i}}$
for $i=0,\infty$ as formal $(p,q)$-Mahler modules over $\widetilde{k}_{i}$,
and $M_{k_{1}}$ as a formal $(p,q)$-difference module over $k_{1}.$
In the latter case we use the variable $z$, in terms of which $\sigma(z)=pz$
and $\tau(z)=qz.$
\begin{itemize}
\item \textbf{Step I.} After writing $x$ for $x^{1/s}$ if necessary, there
exist matrices $C_{i}\in GL_{r}(k_{i})$ ($i=0,1,\infty)$ such that
\[
\begin{cases}
\begin{array}{c}
\sigma(C_{i})^{-1}AC_{i}=A_{i}\\
\tau(C_{i})^{-1}BC_{i}=B_{i}
\end{array}\end{cases}
\]
are constant matrices.
\end{itemize}
By Theorems \ref{thm:2Q formal structure} and \ref{thm:Formal q-Mahler filtration}
there exist such matrices over $\widetilde{k}_{0},\widetilde{k}_{\infty}$
and $k_{1}.$ After a change of variable, substituting $x$ for $x^{1/s}$,
we may assume that $C_{0}$ and $C_{\infty}$ have entries in $k_{0}$
and $k_{\infty}.$ Such a move leaves the field $k_{1}$ unchanged.
\begin{itemize}
\item \textbf{Step II. }We may assume that $C_{i}\in GL_{r}(\mathcal{O}_{i})$
where $\mathcal{O}_{i}$ is the ring of integers of $k_{i},$ and
\[
C_{i}\equiv I\mod t_{i}
\]
\end{itemize}
A global gauge transformation replaces $(A,B)$ by $(\sigma(C){}^{-1}AC,\tau(C)^{-1}BC)$
and $C_{i}$ by $C^{-1}C_{i}$ for some $C\in GL_{r}(K).$ The constant
matrices $(A_{i},B_{i})$ are unchanged. By weak approximation in
the field $K$, we may find a $C\in GL_{r}(K)$ such that $C^{-1}C_{i}$
are, simultaneously, as close as we wish to $I\in GL_{r}(k_{i}),$
and in particular are in the open set $GL_{r}(\mathcal{O}_{i}),$
and congruent to $I$ modulo the maximal ideal.

Observe that once $C_{i}\in GL_{r}(\mathcal{O}_{i}),$ then also $A\in GL_{r}(\mathcal{O}_{i}).$
Since $A$ is meromorphic, it is holomorphic at the point $i$. The
same applies to the matrix $B$. Furthermore, the assumption $C_{i}\equiv I\mod t_{i}$
implies $A\equiv A_{i}\mod t_{i},$ $B\equiv B_{i}\mod t_{i}.$ 
\begin{itemize}
\item \textbf{Step III. }Each $C_{i}$ is holomorphic at some neighborhood
of the point $i=0,1,\infty.$
\end{itemize}
To prove this we use estimates on the coefficients in the formal Taylor
exapnsion. For example, at $i=0,$ write
\[
C_{0}(x)=M_{0}+M_{1}x+M_{2}x^{2}+\cdots,\,\,\,A(x)=A_{0}+N_{1}x+N_{2}x^{2}+\cdots
\]
($M_{0}=I$). From $A(x)C_{0}(x)=C_{0}(x^{p})A_{0}$ we get the recursion
formula
\[
A_{0}M_{n}+\sum_{i=1}^{n}N_{i}M_{n-i}=M_{n/p}A_{0}
\]
where $M_{n/p}=0$ if $p$ does not divide $n$. Let $||.||$ be any
norm on the space of $r\times r$ complex matrices (they are all equivalent).
The analyticity of $A(x)$ at $0$ implies that there exists a $c_{1}>0$
such that $||N_{i}||<c_{1}^{i}.$ It follows easily from this and
from the recursion formula that $||M_{i}||<c_{2}^{i}$ for some $c_{2}>0,$
hence that $C_{0}(x)$ converges absolutely in $|x|<c_{2}^{-1}.$
The point $i=\infty$ is treated similarly. 

At $i=1,$ using expansions in the local parameter $z=\log(x),$ we
have
\[
C_{1}(x)=C_{1}(e^{z})=\widetilde{C}_{1}(z)=M_{0}+M_{1}z+M_{2}z^{2}+\cdots,\,\,\,A(e^{z})=A_{1}+N_{1}z+N_{2}z^{2}+\cdots
\]
$(M_{0}=I).$ From $A(e^{z})\widetilde{C}_{1}(z)=\widetilde{C}_{1}(pz)A_{1}$
we now get the recursion formula
\[
M_{n}p^{n}-A_{1}M_{n}A_{1}^{-1}=\sum_{i=1}^{n}N_{i}M_{n-i}A_{1}^{-1}.
\]
Since for $n>>0$
\[
0.9p^{n}||M_{n}||\le||M_{n}p^{n}-A_{1}M_{n}A_{1}^{-1}||\le1.1p^{n}||M_{n}||,
\]
we may conclude the proof of step III as before.
\begin{itemize}
\item \textbf{Step IV. }The matrix $C_{0}(x)$ admits meromorphic continuation
to $0\le|x|<1,$ and the matrix $C_{\infty}(x)$ admits meromorphic
continuation to $1<|x|\le\infty.$
\end{itemize}
The functional equation $C_{0}(x)=A(x)^{-1}C_{0}(x^{p})A_{0}$ shows
that if $C_{0}(x)$ has meromorphic continuation to the disk $D(0,r)$
for some $r<1$, then it has such a meromorphic continuation to $D(0,r^{1/p}).$
Since for $r$ small enough $C_{0}(x)$ is in fact holomorphic in
$D(0,r),$ the claim follows. The same argument holds at $\infty.$
\begin{itemize}
\item \textbf{Step V.} The matrix $C_{0}(x)$ admits meromorphic continuation
to $0\le|x|<\infty.$ Similarly $C_{\infty}(x)$ admits meromorphic
continuation to $0<|x|\le\infty.$
\end{itemize}
Crossing the natural boundary at $|x|=1$ is subtle. This is where
the expansion around $i=1$ comes to our rescue. Recall that $C_{1}(x)$
is a-priori defined and analytic only in $|x-1|<\varepsilon$ for
some $0<\varepsilon$. Trying to use \emph{one} of the two functional
equations
\[
\begin{cases}
\begin{array}{c}
A_{1}=C_{1}(x^{p})^{-1}A(x)C_{1}(x)\\
B_{1}=C_{1}(x^{q})^{-1}B(x)C_{1}(x)
\end{array}\end{cases}
\]
to meromorphically continue it to $0<|x|<\infty$ as we did with $C_{0}$
or $C_{\infty}$ leads to issues of monodromy. The key idea, due to
\cite{Sch-Si1,Sch-Si2}, is to use \emph{both} functional equations
to overcome the monodromy. The arguments below constitute a slight
variation on the original arguments.

Write $x=e^{z}$ and define, for $i=0,1$,
\[
\widetilde{C}_{i}(z)=C_{i}(e^{z}).
\]
By the previous step, $\widetilde{C}_{0}(z)$ is meromorphic in $\{Re(z)<0\}$
and is $2\pi i$-periodic there, while $\widetilde{C}_{1}(z)$ is
a-priori defined and analytic only in a neighborhood of $z=0.$ The
functional equation
\[
A_{1}=\widetilde{C}_{1}(pz)^{-1}A(e^{z})\widetilde{C}_{1}(z)
\]
gives a meromorphic continuation of $\widetilde{C}_{1}(z)$ to all
$z\in\mathbb{C}.$

Let
\[
\widetilde{C}_{01}(z)=\widetilde{C}_{0}(z)^{-1}\widetilde{C}_{1}(z)\,\,\,\,(Re(z)<0).
\]
Then
\[
\begin{cases}
\begin{array}{c}
\widetilde{C}_{01}(pz)=A_{0}\widetilde{C}_{01}(z)A_{1}^{-1}\\
\widetilde{C}_{01}(qz)=B_{0}\widetilde{C}_{01}(z)B_{1}^{-1}.
\end{array}\end{cases}
\]

\begin{lem}
There exist $\pi/2<\alpha<\beta<3\pi/2$ such that $\widetilde{C}_{01}(z)$
is analytic in the sector $\{\alpha<Arg(z)<\beta\}.$
\end{lem}

\begin{proof}
Since the poles of $\widetilde{C}_{01}(z)$ have no accumulation point
in $\{Re(z)<0\},$ there are no poles in
\[
S=\{z|\,p^{-1}\le|z|\le p,\,\,\,\alpha<Arg(z)<\beta\},
\]
for suitable $\pi/2<\alpha<\beta<3\pi/2$. The relation $\widetilde{C}_{01}(pz)=A_{0}\widetilde{C}_{01}(z)A_{1}^{-1}$
now yields the lemma.
\end{proof}
\begin{lem}
$\widetilde{C}_{0}(z)$ has a meromorphic continuation to all $z\in\mathbb{C}$
and is $2\pi i$-periodic.
\end{lem}

\begin{proof}
Assume that we prove
\begin{itemize}
\item $\widetilde{C}_{01}(z)$ has an analytic continuation to $\mathbb{C}\setminus[0,\infty).$
\end{itemize}
Then $\widetilde{C}_{0}(z)=\widetilde{C}_{1}(z)\widetilde{C}_{01}(z)^{-1}$
also admits a meromorphic continuation to $\mathbb{C}\setminus[0,\infty).$
In $\{Re(z)<0\},$ $\widetilde{C}_{0}(z)$ was $2\pi i$-periodic.
It is therefore $2\pi i$-periodic in the upper half plane, so extends
by periodicity to the whole complex plane, and the lemma is verified.

To prove that $\widetilde{C}_{01}(z)$ has an analytic continuation
to $\mathbb{C}\setminus[0,\infty)$ consider
\[
D(w)=\widetilde{C}_{01}(e^{w}),
\]
a-priori analytic in the strip $\{\alpha<Im(w)<\beta\}.$ It satisfies
there
\[
\begin{cases}
\begin{array}{c}
D(w+\log(p))=A_{0}D(w)A_{1}^{-1}\\
D(w+\log(q))=B_{0}D(w)B_{1}^{-1}
\end{array}.\end{cases}
\]
Recall that $A_{0}$ and $B_{0}$ commute, and so do $A_{1}$ and
$B_{1}.$ Let $L_{0}$ and $L_{1}$ be matrices commuting with $B_{0}$
and $B_{1}$ respectively, such that
\[
A_{0}=e^{L_{0}},\,\,\,A_{1}=e^{L_{1}}.
\]
Then
\[
E(w)=e^{-L_{0}w/\log p}D(w)e^{L_{1}w/\log p}
\]
satisfies

\[
\begin{cases}
\begin{array}{c}
E(w+\log(p))=E(w)\\
E(w+\log(q))=F_{0}E(w)F_{1}^{-1}
\end{array},\end{cases}
\]
with
\[
F_{i}=B_{i}e^{-L_{i}\log q/\log p}.
\]
The first equation implies that in the strip $\{\alpha<Im(w)<\beta\}$
we have a Fourier expansion
\[
E(w)=\sum_{n\in\mathbb{Z}}E_{n}e^{2\pi inw/\log p}
\]
with constant matrices $E_{n}.$ The second equation then implies
\[
E_{n}e^{2\pi in\log q/\log p}=F_{0}E_{n}F_{1}^{-1}.
\]
As the linear transformation $M\mapsto F_{0}MF_{1}^{-1}$ on the space
of $r\times r$ matrices can have only finitely many eigenvalues,
and as the numbers $e^{2\pi in\log q/\log p}$ are all distinct (thanks
to the multiplicative independence of $p$ and $q$), we deduce that
$E_{n}=0$ for all but finitely many $n.$ This shows that $E(w),$
and with it $D(w),$ are entire.

Going back to the definition of $D(w)$ we conclude that $\widetilde{C}_{01}(z)=D(\log(z))$,
a-priori only analytic in the sector $\alpha<Arg(z)<\beta$, has an
analytic continuation to $\mathbb{C}\setminus[0,\infty)$. Here $\log(z)$
denotes the principal branch of $\log$ on the complement of $[0,\infty).$
This concludes the proof of the lemma.
\end{proof}
The Lemma clearly implies that $C_{0}(x)$ may be meromorphically
continued all the way to the north pole. The matrix $C_{\infty}(x)$
is treated in the same way. The proof of Theorem \ref{thm:Main Theorem}
is concluded with the following (last) step.
\begin{itemize}
\item \textbf{Step VI. }The matrix $C_{0}(x)\in GL_{r}(K).$
\end{itemize}
Consider $C_{0\infty}(x)=C_{0}(x)^{-1}C_{\infty}(x).$ This matrix
is meromorphic in $0<|x|<\infty$ and satisfies
\[
C_{0\infty}(x^{p})=A_{0}C_{0\infty}(x)A_{\infty}^{-1}.
\]
In any annulus $V(r_{1},r_{2})=\{r_{1}<|x|<r_{2}\}$ where $C_{0\infty}(x)$
is analytic, it has a power series expansion
\[
C_{0\infty}(x)=\sum_{n\in\mathbb{Z}}M_{n}x^{n}.
\]
The functional equation relates the expansions on $V(r_{1},r_{2})$
and $V(r_{1}^{p},r_{2}^{p}),$ where $C_{0\infty}(x)$ is also analytic,
and shows that $M_{n}=0$ unless $p|n$. Iterating, we see that $M_{n}=0$
unless $p^{k}|n$ for $k=1,2,\dots.$ It follows that the only non-zero
coefficient is $M_{0}$ and $C_{0\infty}$ is constant. Since $C_{\infty}(x)$
is meromorphic (in fact analytic) at $\infty,$ $C_{0}(x)$ is meromorphic
everywhere on $\mathbb{P}^{1}(\mathbb{C})$, hence is a matrix of
rational functions.

\subsection{$(p,q)$-difference modules\label{subsec:(p,q)-difference modules}}

In this section we prove Theorem \ref{thm:Main Theorem} in the case
2Q. Recall that $K=\mathbb{C}(x)$ and the two difference operators
are
\[
\sigma(x)=px,\,\,\,\tau(x)=qx,
\]
where $p$ and $q$ are multiplicatively independent non-zero complex
numbers. We make the following assumption:
\begin{itemize}
\item (Hyp) At least one of $p$ or $q$ is of absolute value $\ne1.$
\end{itemize}
Without loss of generality (replacing $\sigma$ by $\sigma^{-1}$
or by $\tau^{\pm1}$, and afterwards replacing $\tau$ by $\tau^{-1}$,
if necessary), we may assume that $|p|>1$ and $|q|\ge1.$ At the
end of the proof we shall explain how to eliminate (Hyp). 

Let $\Gamma=\left\langle \sigma,\tau\right\rangle \subset Aut(K)$.
Then $\Gamma$ is free abelian of rank $2$. Let $M$ be a rank $r$
$\Gamma$-difference module over $K$ (called also a $(p,q)$-difference
module). Fix a basis of $M$ and let $A=A_{\sigma}$ and $B=A_{\tau}$
be the matrices attached to $\Phi_{\sigma}$ and $\Phi_{\tau}$ in
this basis as in \S\ref{subsec:Matrices}.

For $i=0$ or $\infty$ let $t_{0}=x$ and $t_{\infty}=1/x$ be local
parameters at the point $i$. Let $k_{i}=\mathbb{C}((t_{i}))$ be
the completion of $K=\mathbb{C}(x)$ at the point $i$. By base-change
we may regard $M_{k_{i}}$ for $i=0,\infty$ as formal $(p,q)$-diffrence
modules over $k_{i}.$ The proofs of the first three steps below are
exactly the same as in the case 2M, so we omit them.
\begin{itemize}
\item \textbf{Step I. }There exist matrices $C_{i}\in GL_{r}(k_{i})$ ($i=0,\infty)$
such that
\[
\begin{cases}
\begin{array}{c}
\sigma(C_{i})^{-1}AC_{i}=A_{i}\\
\tau(C_{i})^{-1}BC_{i}=B_{i}
\end{array}\end{cases}
\]
are constant matrices.
\item \textbf{Step II. }We may assume that $C_{i}\in GL_{r}(\mathcal{O}_{i})$
where $\mathcal{O}_{i}$ is the ring of integers of $k_{i},$ and
\[
C_{i}\equiv I\mod t_{i}.
\]
\item \textbf{Step III. }$C_{i}$ is holomorphic at some neighborhood of
the point $i=0,\infty.$
\item \textbf{Step IV. }The matrix $C_{0}(x)$ admits meromorphic continuation
to $0\le|x|<\infty,$ and the matrix $C_{\infty}(x)$ admits meromorphic
continuation to $0<|x|\le\infty.$
\end{itemize}
As before, we use the functional equation $C_{0}(px)=A(x)C_{0}(x)A_{0}^{-1}$
to meromorphically continue $C_{0}(x)$ from $D(0,r)$ to $D(0,pr).$
Here the assumption $|p|>1$ is used. A similar argument works for
$C_{\infty}(x).$
\begin{itemize}
\item \textbf{Step V.} The matrix $C_{0}(x)\in GL_{r}(K).$
\end{itemize}
Consider the functional equations
\[
\begin{cases}
\begin{array}{c}
C_{0}(px)=A(x)C_{0}(x)A_{0}^{-1}\\
C_{0}(qx)=B(x)C_{0}(x)B_{0}^{-1}
\end{array} & .\end{cases}
\]
Let $R$ be large enough so that $A(x)$ and $B(x)$ and their inverses
have no poles in $U=\{R<|x|<\infty\}.$ Let $S$ be the set of poles
of $C_{0}(x)$ in $U$. The functional equations imply that if $x$
and $px$, or $x$ and $qx$, are both in $U,$ then they are either
both in $S$ or both not in $S.$ As $p$ and $q$ are multiplicatively
independent, we see that if $S$ is not empty then for a suitable
$R'$ the compact subset $Z=\{R'\le|x|\le p^{2}R'\}\subset U$ contains
infinitely many distinct points of the form $p^{-a}q^{b}x_{0}$ for
some $x_{0}\in S$ and $a,b\ge0.$ Indeed, if $|q|=1$ we may take
the points $q^{b}x_{0}$ where $R'$ is chosen so that $Z\cap S$
is non-empty, and $x_{0}\in Z\cap S.$ If $|q|>1$ we take $q^{b}x_{0}$
($b\ge0$) and then find for each $b$ an $a$ such that $p^{-a}q^{b}x_{0}\in Z.$
This implies however that $Z$ contains infinitely many points in
$S$. It follows that $S$ is empty, and $C_{0}(x)$ is analytic in
$U.$

Consider, as in the case 2M, the function $C_{0\infty}(x)=C_{0}(x)^{-1}C_{\infty}(x).$
By choosing $R$ large enough we see that $C_{0\infty}(x)$ is analytic
in $U,$ so admits there a power series expansion
\[
C_{0\infty}(x)=\sum_{n\in\mathbb{Z}}M_{n}x^{n}.
\]
Furthermore, it satisfies in $U$ the functional equation
\[
C_{0\infty}(px)=A_{0}C_{0\infty}(x)A_{\infty}^{-1},
\]
implying $p^{n}M_{n}=A_{0}M_{n}A_{\infty}^{-1}.$ As the linear transformation
$M\mapsto A_{0}MA_{\infty}^{-1}$ can have only finitely many eigenvalues,
$M_{n}=0$ for all but finitely many values of $n$. It follows that
$C_{0\infty}(x)$, and with it $C_{0}(x),$ is meromorphic at $\infty.$
Thus the entries of $C_{0}(x)$ are everywhere meromorphic on $\mathbb{P}^{1}(\mathbb{C}),$
so belong to $K=\mathbb{C}(x).$ This concludes the proof of the last
step, and with it of the main theorem, under the assumption (Hyp).
\begin{itemize}
\item \textbf{Step VI. }Elimination of the assumption (Hyp).
\end{itemize}
As we have seen in Remark (iv) following Theorem \ref{thm:Main Theorem},
while the proof of Step IV above used the dynamics of $z\mapsto pz$
(namely the fact that by iterating this map an arbitrarily small open
neighborhood of $0$ eventually covered the whole of $\mathbb{C}$),
the statement of the Main Theorem is purely algebraic. Thus (Hyp)
can be weakened to assume that \emph{under some abstract automorphism}
$\iota$ of $\mathbb{C}$ one of $\iota(p)$ or $\iota(q)$ does not
lie on the unit circle. There are still algebraic numbers for which
this can not be achieved. For example, if $E$ is a CM field and $p=P/\overline{P},$
$q=Q/\overline{Q}$ for some $P,Q\in E$ (it is an easy exercise that
we can make such $p$ and $q$ multiplicatively independent).

However, let $\ell$ be an auxiliary rational prime, let $\mathbb{\mathbb{C}_{\ell}}$
be the completion of an algebraic closure of $\mathbb{Q}_{\ell}$,
and consider an abstract algebraic isomorphism
\[
\iota:\mathbb{C}\simeq\mathbb{C}_{\ell}.
\]
Such a $\iota$ exists because both fields have the same transcendence
cardinality and are algebraically closed. Now, the entire proof given
above works, \emph{mutatis mutandis, }over $\mathbb{C}_{\ell}$ instead
of $\mathbb{C},$ provided (Hyp) is replaced by (Hyp$_{\ell}$): \emph{At
least one of $\iota(p)$ or $\iota(q)$ is of absolute value $\ne1$.
}One should understand ``analytic'' or ``meromorphic'' in the
rigid analytic sense. Note that the only step where Calculus was used
was Step III, and this step becomes even easier over $\mathbb{C}_{\ell}$
thanks to the ultrametric inequality.

It follows that the only case not covered by the above proof is when
$p$, and similarly $q$, maps to the unit circle under any field
isomorphism $\iota:\mathbb{C}\simeq\mathbb{C}_{\ell}$ for any prime
$\ell,$ including $\infty$. It is well-known that this happens if
and only if $p$ and $q$ are both roots of unity, a case ruled out
by the assumption on multiplicative independence.

\section{$p$-Mahler $q$-difference modules\label{sec:-1M1Q}}

In this part we illustrate the same approach used in cases 2M and
2Q in a third example, where the group $\Gamma$ is generated by one
$q$-difference operator and one Mahler operator, and turns out to
be generalized dihedral. We therefore call this \emph{Case} 1M1Q.

\subsection{Formal $p$-Mahler $q$-difference modules}

\subsubsection{The group $\Gamma$}

Let $\widetilde{K}$ be as before, let $p\ge2$ be a natural number
and $q\in\mathbb{C}^{\times}$ a complex number which is not a root
of unity. No assumption of independence is made on $p$ and $q.$
Fix a compatible sequence of roots $q^{1/s}$ as before.

Let $\Gamma=\left\langle \sigma,\tau\right\rangle \subset Aut(\widetilde{K})$
where
\[
\sigma(x^{1/s})=x^{p/s},\,\,\,\tau(x^{1/s})=q^{1/s}x^{1/s}.
\]
The easily verified relation
\[
\sigma^{-1}\circ\tau\circ\sigma=\tau^{p}
\]
yields
\[
\Gamma\simeq\mathbb{Z}\ltimes\mathbb{Z}[\frac{1}{p}]
\]
where $(-n,0)(0,a)(n,0)=(0,p^{n}a)$. Here $\sigma\mapsto(1,0)$ and
$\tau\mapsto(0,1).$ Thus $\Gamma$ is \emph{generalized dihedral}
rather than abelian.
\begin{lem}
Every element of $\Gamma$ is of the form $\sigma^{a}\tau^{b}\sigma^{c}$
for $a,b,c\in\mathbb{Z}.$
\end{lem}

\begin{proof}
Every element of $\Gamma$ is of the form $\sigma^{i}\tau^{j/p^{n}}$
for some $i,j,n\in\mathbb{Z}.$ But
\[
\sigma^{i}\tau^{j/p^{n}}=\sigma^{i+n}\tau^{j}\sigma^{-n}.
\]
\end{proof}

\subsubsection{$p$-Mahler $q$-difference modules}

We shall call a $\Gamma$-difference module over $\widetilde{K}$
(or $\widetilde{k})$ a $p$-Mahler $q$-difference module. Let $f\in\widetilde{k}$,
and assume that the $\widetilde{K}\left\langle \Gamma\right\rangle $-submodule
of $\widetilde{k}$ generated by $f$ is finite dimensional over $\widetilde{K}.$
Then this module is a $p$-Mahler $q$-difference module over $\widetilde{K}$,
and arguments similar to those of Proposition \ref{prop:Modules_and_equations}
may be applied.

We label this new case by 1M1Q. Unlike cases 2M and 2Q, for $M=\widetilde{K}\left\langle \Gamma\right\rangle f$
to be finite dimensional over $\widetilde{K}$, the (necessary) condition
that $f$ satisfies both a $\sigma$-Mahler equation and a $\tau$-difference
equation is not sufficient. This is beacuse $\Gamma$ is not abelian
anymore.

The best we can say with regard to \emph{equations} is that since
every element of $\Gamma$ is of the form $\sigma^{a}\tau^{b}\sigma^{c}$
for $a,b,c\in\mathbb{Z},$ a \emph{finite} number of equations will
suffice to guarantee $\dim_{\widetilde{K}}M<\infty$. One will need,
for example, a $\sigma$-Mahler equation for $f$, say of degree $n,$
then for each $0\le c\le n-1$ a $\tau$-difference equation for $\sigma^{c}f,$
and if $m,$ say, is the maximum of the degrees of these equations,
for each $0\le c\le n-1$ and $0\le b\le m-1$ a $\sigma$-Mahler
equation for the power series $\tau^{b}\sigma^{c}f.$ This collection
of equations will guarantee that the elements $\sigma^{a}\tau^{b}\sigma^{c}f,$
for $a,b,c$ in a bounded range, will span $M$ over $\widetilde{K}.$

The discussion above makes it clear that for our generalized dihedral
$\Gamma,$ the natural condition for a result in the style of Theorems
\ref{thm:L-vdP} and \ref{thm:Bezivin} is the finite dimensionality
of $M=\widetilde{K}\left\langle \Gamma\right\rangle f$. Its formulation
in terms of equations can be cumbersome.

\subsubsection{Formal $p$-Mahler $q$-difference modules}
\begin{thm}
\label{thm:formal 1M1Q}Let $M$ be a $p$-Mahler $q$-difference
module over $\widetilde{k}$. Then $M$ has a unique $\mathbb{C}$-structure
$M_{0}$ preserved by both $\Phi_{\sigma}$ and $\Phi_{\tau}$, such
that $\Phi_{\tau}$ acts potentially unipotently on $M_{0}$.
\end{thm}

\begin{proof}
Suppose $\lambda\in\mathbb{Q}$ is a slope of $M$, considered as
a $q$-difference module. Then there exists a $c\in\mathbb{C}^{\times}$,
uniquely determined up to multiplication by $q^{\alpha},$ $\alpha\in\mathbb{Q},$
and a $0\ne v\in M$, such that $\Phi_{\tau}v=cx^{\lambda}v.$ Since
$\Phi_{\tau}^{p}v=c^{p}q^{\binom{p}{2}\lambda}x^{p\lambda}v,$ the
equation $\Phi_{\tau}\circ\Phi_{\sigma}=\Phi_{\sigma}\circ\Phi_{\tau}^{p}$
yields
\[
\Phi_{\tau}(\Phi_{\sigma}v)=c^{p}q^{\binom{p}{2}\lambda}x^{p^{2}\lambda}\cdot\Phi_{\sigma}v.
\]
It follows that $p^{2}\lambda$ is also a slope of $M$ as a $q$-difference
module. We can repeat this argument, and since the number of slopes
is finite, $\lambda=0.$ This means that $M$ descends to $\mathbb{C},$
i.e. $M=\widetilde{k}\otimes_{\mathbb{C}}M_{0},$ as a $q$-difference
module. Furthermore, if $c$ is an eigenvalue of $\Phi_{\tau}$ on
$M,$ the above computation shows that so is $c^{p}.$ Since there
are only finitely many eigenvalues modulo $q^{\mathbb{Q}}$ it follows
that for some $m>n\ge1$ and $\alpha\in\mathbb{Q}$ we must have
\[
c^{p^{m}}=c^{p^{n}}q^{\alpha}.
\]
This means that $c=\zeta q^{\mu}$ for some rational number $\mu$,
and a root of unity $\zeta.$ Let $M_{0}(c)$ be the direct summand
of $M_{0}$ with generalized $\Phi_{\tau}$-eigenvalue $c.$ Replacing
it by its ``resonant'' $x^{-\mu}M_{0}(c),$ we may assume that $c=\zeta.$
Going over all the eigenvalues of $\Phi_{\tau}$ on $M_{0}$ in this
way, we may assume that they have all been replaced by roots of unity,
so some power $\Phi_{\tau}^{m}$ acts unipotently on $M_{0}.$ This
pins down $M_{0},$ namely
\[
M_{0}=\left\{ v\in M|\,\exists n\,(\Phi_{\tau}^{m}-1)^{n}v=0\right\} .
\]
Substituting $\zeta$ for $c$ in the computation above we see that
$\Phi_{\sigma}$ preserves $M_{0}[\Phi_{\tau}^{m}-1]$, hence by dévissage
preserves also $M_{0}.$ This concludes the proof of the theorem.
\end{proof}

\subsection{Rational $p$-Mahler $q$-difference modules}

The analogue of Theorem \ref{thm:Main Theorem} in case 1M1Q is the
following.
\begin{thm}
Let $M$ be a $p$-Mahler $q$-difference module over $\widetilde{K}$.
Then $M$ has a unique $\mathbb{C}$-structure $M_{0}$ preserved
by both $\Phi_{\sigma}$ and $\Phi_{\tau}$, such that $\Phi_{\tau}$
acts potentially unipotently on $M_{0}$.
\end{thm}

\begin{proof}
As in cases 2Q and 2M, choose a basis of $M$ over $\widetilde{K}$
and let $A=A_{\sigma}$ and $B=A_{\tau}$ represent $\Phi_{\sigma}$
and $\Phi_{\tau}$ in this basis. Our goal is to show that the pair
$(A,B)$ is gauge-equivalent to a pair of constant matrices $(A_{0},B_{0}),$
and moreover that all the eigenvalues of $B_{0}$ are roots of unity.

Without loss of generality we may assume that $|q|>1$. The reduction
to this case is done precisely as in case 2Q; see step VI in \S\ref{subsec:(p,q)-difference modules},
\emph{elimination of the assumption (Hyp)}.

We consider the points $i=0,\infty$ and proceed as in case 2Q. Invoking
theorem \ref{thm:formal 1M1Q} and repeating the arguments in steps
I-IV there we get:

\textbf{Steps I-IV: }After a change of variables, writing $x$ for
$x^{1/s}$ for a suitable $s$, there exists an invertible matrix
$C_{0}(x)$, meromorphic in $0\le|x|<\infty$ and holomorphic at 0,
and constant matrices $A_{0},B_{0}$, such that the following equations
hold
\[
\begin{cases}
\begin{array}{c}
C_{0}(x^{p})^{-1}A(x)C_{0}(x)=A_{0}\\
C_{0}(qx)^{-1}B(x)C_{0}(x)=B_{0}
\end{array} & .\end{cases}
\]
Furthermore, all the eigenvalues of $B_{0}$ are roots of unity.

Likewise, there exists an invertible matrix $C_{\infty}(x),$ meromorphic
in $0<|x|\le\infty$ and holomorphic at $\infty,$ and constant matrices
$A_{\infty},B_{\infty},$ such that
\[
\begin{cases}
\begin{array}{c}
C_{\infty}(x^{p})^{-1}A(x)C_{\infty}(x)=A_{\infty}\\
C_{\infty}(qx)^{-1}B(x)C_{\infty}(x)=B_{\infty}
\end{array} & .\end{cases}
\]

\textbf{Step V: }The matrix $C_{0}(x)\in GL_{r}(K).$

Consider
\[
C_{0\infty}(x)=C_{\infty}(x)^{-1}C_{0}(x),
\]
which is meromorphic in $0<|x|<\infty.$ It satisfies there the functional
equation
\[
C_{0\infty}(x^{p})A_{0}=A_{\infty}C_{0\infty}(x).
\]
Arguing as in Step VI in case 2M, on the power-series expansions of
$C_{0\infty}(x)$ in annuli of analyticity, we deduce that $C_{0\infty}$
is constant. It follows that $C_{0}(x)$ is meromorphic also at $\infty,$
hence is rational.

This concludes the proof of the theorem.
\end{proof}
As in Proposition \ref{prop:Modules_and_equations} we can withdraw
from the last theorem the following consequence.
\begin{thm}
Let $f\in\mathbb{C}((x))$ and assume that $f\in M\subset\widetilde{k}=\bigcup_{s\in\mathbb{N}}\mathbb{C}((x^{1/s})),$
where $M$ is a finite dimensional $\widetilde{K}$-vector space closed
under $\sigma$ and $\tau$. Then $f\in\mathbb{C}(x).$
\end{thm}

Note that the assumption on the finite dimensionality of $M$ replaces
the (insufficient) assumption that $f$ satisfies a $p$-Mahler equation
and a $q$-difference equation simultaneously. As remarked before,
it is possible to encode this assumption in a finite number of equations,
but their number will depend, in general, on the power series $f$,
and they will be of mixed type, iterations of both $\sigma$ and $\tau$
figuring in the same equation.

\section{Finite characteristic}

In this section we briefly explain how to modify the proof of Theorem
\ref{thm:Bezivin} and Theorem \ref{thm:Main Theorem} (in the case
2Q), when $\mathbb{C}$ is replaced by an arbitrary algebraically
closed field. We thus prove the following.
\begin{thm}
Theorems \ref{thm:Bezivin} and \ref{thm:Main Theorem} (case 2Q)
remain valid as stated, when $\mathbb{C}$ is replaced by an arbitrary
algebraically closed field $C$.
\end{thm}

\begin{proof}
If $C$ has characteristic $0$ one can apply the Lefschetz principle
and assume it is $\mathbb{C}$. Let therefore $char.(C)=\ell>0.$
The proof of Proposition \ref{prop:Modules_and_equations}, deducing
Theorem \ref{thm:Bezivin} from Theorem \ref{thm:Main Theorem}, did
not use any property of $\mathbb{C},$ besides it being a field. We
therefore only have to explain how to modify the proof of Theorem
\ref{thm:Main Theorem}.

Theorem \ref{thm:2Q formal structure}, giving the structure of a
formal $(p,q)$-difference module, also did not use any property of
the field of constants, and works equally well if $C$ has finite
characteristic. This provides the starting point for the proof, and
steps I-II of \S\ref{subsec:(p,q)-difference modules} hold true
with $C$ replacing $\mathbb{C}.$ We now use the following lemma.
\begin{lem}
Let $C$ be an algebraically closed field of characteristic $\ell$
and $p\in C^{\times}$ not a root of unity. Then there exists an algebraically
closed complete valued field $(\widehat{C},|.|)$ containing $C,$
such that $|p|>1$.
\end{lem}

\begin{proof}
As $p$ is transcendental over $\mathbb{F}_{\ell}$ we can complete
it to a transcendental basis $\{z_{\alpha}\}$ of $C$ over $\mathbb{F}_{\ell}$,
with $z_{0}=p$. Let $F$ be the field generated over $\mathbb{F}_{\ell}$
by the $z_{\alpha}$ for $\alpha\ne0$, and consider $F(z_{0})$ with
a valuation which is trivial on $F$ and satisfies $|z_{0}|>1$. Let
$\widehat{C}$ be an algebraically closed complete extension of $F(z_{0})$
to which $|.|$ extends. Since $C$ is algebraic over $F(z_{0}),$
it embeds in $\widehat{C}$.
\end{proof}
We continue as in \S\ref{subsec:(p,q)-difference modules}, reserving
the terms ``holomorphic'' and ``meromorphic'' to mean ``rigid
holomorphic (resp. meromorphic) over $\widehat{C}$''. Steps III-V,
concluding the proof, are carried out now in the same way as over
$\mathbb{C}$, taking advantage of the fact that $|p|>1.$ Compare
with the use of $\mathbb{C}_{\ell}$ to eliminate assumption (hyp)
in characteristic 0, in loc.cit., Step VI.
\end{proof}
\begin{rem}
The extension of cases 2M and 1M1Q to finite characteristic demands
special attention, for the following reason. The substitution $z=\log(x)$,
which allowed us to delegate the formal study of a rational Mahler
module at the fixed point $x=1$ to the realm of $q$-difference modules,
is no longer valid in finite characteristic. In fact, the formal multiplicative
group is not isomorphic to the formal additive group, and therefore
the results of \S\ref{sec:Formal} have to be recast in a new setup.
Notwithstanding this remark, we believe that the main theorems in
case 2M remain valid in finite characteristic $\ell$, at least if
$(\ell,pq)=1.$
\end{rem}

\end{document}